\documentclass[smallextended,numbook]{svjour3}

\smartqed  

\usepackage{algorithm,algorithmic}
\usepackage{cite}
\usepackage{mathrsfs} 
\usepackage{tabularx}
\usepackage{amsmath, amssymb, amstext, microtype} 
\usepackage{graphicx,color,epsfig}
\usepackage{float}
\usepackage{verbatim}
\usepackage{enumerate}
\usepackage[colorlinks,
			linkcolor=blue,
			anchorcolor=blue,
			citecolor=blue]{hyperref}
\numberwithin{equation}{section}

\usepackage{algorithm,algorithmic}

\usepackage{threeparttable}  

\def\R{\mathbb{R}}

\spnewtheorem{prop}[theorem]{Proposition}{\bf}{\it}
\spnewtheorem{assump}[theorem]{Assumption}{\bf}{\it}
\spnewtheorem{lem}[theorem]{Lemma}{\bf}{\it}

\DeclareMathOperator{\trace}{trace}

\DeclareMathOperator*{\argmin}{arg\,min}

\def\cS{{\mathcal{S}}}

\def\R{{\rm I\!R}}

\title{$\rho$-regularization subproblems: Strong duality and an eigensolver-based algorithm}

\author{Liaoyuan Zeng  \and
	Ting Kei Pong\thanks{Ting Kei Pong was supported in part by an internal funding, G-UAKK, of the Hong Kong Polytechnic University.}
}

\institute{Liaoyuan Zeng \at
	Department of Applied Mathematics, The Hong Kong Polytechnic University, Hong Kong, PRC\\
	\email{lyzeng@polyu.edu.hk}
	\and
	Ting Kei Pong \at
	Department of Applied Mathematics, The Hong Kong Polytechnic University, Hong Kong, PRC\\
	\email{tk.pong@polyu.edu.hk}
}

\date{Received: date / Accepted: date}

\begin{document}
	
\maketitle

\begin{abstract}
Trust-region (TR) type method, based on a quadratic model such as the trust-region subproblem (TRS) and $ p $-regularization subproblem ($p$RS), is arguably one of the most successful methods for
unconstrained minimization. In this paper, we study a general regularized subproblem (named $ \rho $RS), which covers TRS and $p$RS as special cases. We derive a strong duality theorem for $ \rho $RS, and also its necessary and sufficient optimality condition under general assumptions on the regularization term.
We then define the Rendl-Wolkowicz (RW) dual problem of $ \rho $RS, which is a maximization problem whose objective function is concave, and differentiable except possibly at \emph{two} points. It is worth pointing out that our definition is based on an \emph{alternative derivation} of the RW-dual problem for TRS. Then we propose an eigensolver-based algorithm for solving the RW-dual problem of $ \rho $RS. The algorithm is carried out by finding the smallest eigenvalue and its unit eigenvector of a certain matrix in each iteration. Finally, we present numerical results on randomly generated $p$RS's, and on a new class of regularized problem that combines TRS and $p$RS, to illustrate our algorithm.
\end{abstract}

\section{Introduction}
The $ p $-regularization and trust-region subproblems arise naturally when using trust-region (TR) type methods for minimizing twice continuously differentiable functions. The $ p $-regularization subproblem ($p$RS)
for unconstrained minimization \cite{GoRoTho10, HsiaSheu17} is defined as
\begin{equation}\label{Prob_prs}
  v_{_{p{\rm RS}}}:= \min_x\  2g^Tx + x^THx + \frac{M}p\|x\|^p,
\end{equation}
where $ p>2 $, $g\in \R^n\backslash\{0\}$, $H\in \cS^n$, the space of symmetric
$n\times n$ matrices, and $M > 0$; while the trust-region subproblem (TRS) is given by
\begin{equation}\label{Prob_trs}
  \begin{array}{rl}
    v_{_{\rm TRS}}:=\min\limits_x &  2g^Tx + x^THx\\
    {\rm s.t.}& \|x\|^2 \le s,
  \end{array}
\end{equation}
with $s>0$. In the literature, the most common choice for $p$ in \eqref{Prob_prs} is $ p=3 $, which corresponds to the cubic-regularization subproblem in \cite{MR2229459, MR2776701}. When the above $p$RS or TRS arises from the minimization of a twice continuously differentiable function, the $g$ and $H$ typically correspond to the (nonzero) gradient and
the Hessian approximation, respectively; see, for example,~\cite{Griewank81, ConGouToi:00}.

The above quadratic models are seen to be extremely successful for TR type methods. They are using some specific regularizations for $\|x\|^2$, namely, \eqref{Prob_trs} utilizes $ \delta_{(-\infty,s]}(t) $, the indicator function of the interval
$(-\infty,s]$,
and the $ p $-regularized term $\frac{M}{p}t_+^{\frac{p}2}$ is applied in \eqref{Prob_prs}.
In this paper, we study the following \emph{more general}, possibly {\em higher order},
$\rho$-regularization subproblem ($\rho$RS):
\begin{equation}
\label{Prob_general}
  v_{\rho r}:=\inf_x\  2g^Tx + x^THx + \rho(\|x\|^2),
\end{equation}
where  $g\in \R^n\backslash\{0\}$, $H\in \cS^n$ and
$\rho:\R\to \R_+\cup\{\infty\}$
is a proper closed convex function with $\rho(0) = 0$. In our subsequent developments, we also consider the following three extra assumptions on $\rho$.\footnote{We will state explicitly in each of our results which of these assumptions are used.}
\begin{assump}
\label{assump:one}
$\rho$ is nondecreasing, $\rho(t) = 0$ for all $t\le 0$ and there exists $t_0 > 0$ with $\rho(t_0)<\infty$.
\end{assump}
\begin{assump}
\label{assump:two}
 $\rho$ is supercoercive on the nonnegative side, i.e., $\lim_{t\to\infty}\frac{\rho(t)}t = \infty$.
\end{assump}
\begin{assump}
\label{assump:three}
The monotone conjugate, $\rho^+$, is differentiable for $t > 0$.
\end{assump}
\noindent Here, we recall that $\rho^+(u):= \sup_{t\ge 0}\{ut -
\rho(t)\}$; see~\cite[Page~111]{roc70}. Under Assumption~\ref{assump:two}, it is routine to show that $\rho^+$ is finite everywhere and is hence continuous.
The above assumptions are general enough so that \eqref{Prob_general}
includes \eqref{Prob_prs} and \eqref{Prob_trs} as special cases: indeed, one can see that \eqref{Prob_prs} corresponds to
\eqref{Prob_general} with $\rho(t) := \frac{M}{p}t_+^{\frac{p}2}$, while
\eqref{Prob_trs} corresponds to \eqref{Prob_general} with $\rho(t) =
\delta_{(-\infty,s]}(t)$, and Assumptions~\ref{assump:one}-\ref{assump:three} are all satisfied for these two specific $\rho$'s. More concrete examples of $ \rho $ satisfying Assumptions~\ref{assump:one}-\ref{assump:three} are presented in Section~\ref{subsec:rho}.
These general assumptions on $\rho$ allow us to adopt regularized subproblems with ``piecewise" regularization terms such as the sum of indicator function and $ p $-regularizer. Yet
they are specific enough for retaining key properties shared by \eqref{Prob_prs} and \eqref{Prob_trs} that are crucial for the development of efficient algorithms, as we next discuss.

To develop our algorithm for solving $\rho$RS \eqref{Prob_general}, we recall one common key property exploited in the development of efficient algorithms for solving \eqref{Prob_prs} and \eqref{Prob_trs}. That is, the necessary and sufficient conditions for
\emph{global} optimality of these two problems can be derived irrespective of convexity; see,
for example,~\cite[Section~7.2]{ConGouToi:00}
and~\cite[Theorem~2.2]{HsiaSheu17}, respectively. This fact is exploited
in the classical Mor\'{e} and Sorensen (MS)
algorithm~\cite{MoSo:83} for solving \eqref{Prob_trs}. This algorithm
applies the Newton method with backtracking to find a root of the
so-called \textit{secular function}, a modification of the necessary
and sufficient optimality conditions of \eqref{Prob_trs}. An analogue of
this algorithm for solving \eqref{Prob_prs} with $ p=3 $ (cubic-regularization subproblem) can be found
in~\cite[Section~6]{MR2776701}. The MS algorithm uses Cholesky
factorizations for computing the Newton search directions, and this can
be inefficient for large-scale problems. There are two common ways to reduce computational cost for large-scale instances of \eqref{Prob_prs} and \eqref{Prob_trs}.
\begin{itemize}
  \item One way is to approximate the
original problem by a carefully constructed sequence of low-dimensional
problems, and then apply a variant of the MS algorithm to these
low-dimensional instances. This is the strategy used in the generalized
Lanczos trust-region (GLTR) method~\cite{GoLuRoTo,GoRoTho10} for
\eqref{Prob_trs}. It uses the Lanczos procedure for constructing the low
dimensional problems; see also the sequential subspace method
in~\cite{Hager:00}. Methods analogous to GLTR are developed for
cubic-regularization subproblem (i.e., \eqref{Prob_prs} with $p = 3$) and $p$RS \eqref{Prob_prs} recently in~\cite{CarmonDuchi18} and \cite{GoRoTho10, GoSi20}, respectively.
  \item Another way is to use eigensolvers to leverage sparsity in $H$. In the Rendl and
Wolkowicz (RW) algorithm~\cite{FortinWolk:03,ReWo:94} and the
large-scale trust-region subproblem (LSTRS)~\cite{Sor:99,MR2401375} for
\eqref{Prob_trs}, this is done by reformulating TRS into parameterized
eigenvalue problems so that one only needs to compute the smallest
eigenvalue and a corresponding unit eigenvector for several (typically
sparse) matrices to solve a TRS. Eigensolvers that can exploit sparsity are then applied. See also~\cite{Adachi} and~\cite{Lieder19} respectively for solving TRS \eqref{Prob_trs} and cubic-regularization subproblems (i.e., \eqref{Prob_prs} with $p = 3$) as {\em one single} generalized eigenvalue problem. On passing, we note that there is currently no RW-type algorithm for \eqref{Prob_prs}.
\end{itemize}
In this paper, we take the latter approach and develop an eigensolver-based algorithm that can exploit sparsity in $H$ for solving the more general problem \eqref{Prob_general}, under Assumptions~\ref{assump:one}-\ref{assump:three}.

The rest of the paper is organized as follows. We derive in Section~\ref{sec:LD_opt} a concave
maximization problem that enjoys a strong duality relationship with
$\rho$RS \eqref{Prob_general} under Assumption~\ref{assump:one}. Necessary and sufficient optimality
conditions for $\rho$RS are then derived. In this sense, \eqref{Prob_general} is {\em intrinsically convex} under Assumption~\ref{assump:one}. We next discuss the RW-dual problem for $\rho$RS in Section~\ref{sec:RW_dual}. Recall that this special
dual problem was originally defined for TRS in developing the RW algorithm. The RW
algorithm for TRS performs an {unconstrained} maximization of a concave
function that is differentiable except possibly at one point. Moreover,
in each iteration, the function value and (super)gradient of the concave
function can be computed by finding the smallest eigenvalue and a
corresponding unit eigenvector of a certain matrix. In this paper, based
on a new alternative derivation of the RW-dual problem for TRS, we derive,
under Assumptions~\ref{assump:one} and \ref{assump:two}, the corresponding RW-dual problem for $\rho$RS. This sets the stage for our development of an RW algorithm for $\rho$RS.

In Section~\ref{sec:algorithm}, under
Assumptions~\ref{assump:one}-\ref{assump:three}, we develop an algorithm for
solving $\rho$RS by solving its RW-dual derived and defined in
Section~\ref{sec:RW_dual}. First, analogous to the case of TRS, we discuss the so-called easy case and hard case (1 and 2) for $ \rho $RS \eqref{Prob_general}, and show that an {\em explicit} solution of
\eqref{Prob_general} can be computed for hard case 2.
We then show that the RW-dual is a maximization
problem whose objective is a concave function that is differentiable
except possibly at {\em two} points. Moreover, the function value and
(super)gradient of this concave function can be computed by finding the
smallest eigenvalue and a corresponding unit eigenvector of a certain
matrix. We also demonstrate how a solution of \eqref{Prob_general} can be
recovered when the concave objective function is differentiable at its
maximizer, and show that the nondifferentiability case corresponds to hard case 2, in which an explicit solution can be obtained. In Section~\ref{sec:pRS}, we specialize
our results to $p$RS as well as a new regularized problem that uses {\em both} the trust-region constraint and the $p$-regularizer. We also discuss sufficient
conditions for $\rho$ to satisfy Assumption~\ref{assump:three} in that section. Our numerical tests in Section~\ref{sec4} illustrate the efficiency of our algorithm.

\paragraph{Notation.} We use $ I $ to denote the identity matrix whose size should be clear from the context. For a symmetric matrix $H\in {\cal S}^n$, we let $ H^{\dagger} $ denote its pseudoinverse and $ \lambda_{\min}(H) $ denote its smallest eigenvalue. We also let $ {\rm Range}(H) $ and $ {\rm ker}(H) $ denote its range space and null space, respectively. Moreover, we write $ H\succ 0 $ (resp., $ H\succeq 0 $) if $ \lambda_{\min}(H)>0 $ (resp., $ \lambda_{\min}(H)\geq 0 $). For a proper convex function $f$, we let $f^*$ denote its conjugate and $\partial f(x)$ denote its set of subdifferential at any $x$. For two proper convex functions $f$ and $g$, we let $f\square g$ denote their infimal convolution; we refer the readers to \cite{roc70} for the definition of infimal convolution. Finally, for an $ f:\R \rightarrow \R $, we let $ f^\prime(t) $ denote the derivative and $ f^\prime_+(t) $ denote the right-hand derivative at $ t\in \R$, if exists. 

\section{Duality theory for $\rho$RS}

\subsection{Lagrange duality and optimality conditions for $\rho$RS}\label{sec:LD_opt}

In this section, under Assumption~\ref{assump:one}, we show that strong duality holds between $\rho$RS \eqref{Prob_general} and a suitably constructed concave dual maximization problem. Based on this, we then derive (necessary and sufficient) optimality conditions for $\rho$RS. The existence of such optimality conditions suggests that $\rho$RS, though appears to be nonconvex, is intrinsically convex.

\begin{theorem}[Strong duality for Lagrange dual]\label{thm0}
  Consider \eqref{Prob_general} and suppose that Assumption~\ref{assump:one} holds. Consider the following concave maximization problem:
  \begin{equation}\label{L_dual_gen}
    \begin{array}{rl}
      v_{\rho d}:= \sup\limits_\lambda & -g^T(H-\lambda I)^\dagger g - \rho^+(-\lambda)\\
      {\rm s.t.}& \lambda \le 0, H - \lambda I \succeq 0, g\in {\rm Range}(H - \lambda I).
    \end{array}
  \end{equation}
  Then $v_{\rho d} = v_{\rho r}$ and a solution of problem \eqref{L_dual_gen} exists when $v_{\rho r}$ is finite.
\end{theorem}
\begin{proof}
We first recall the following equality for each $t\ge 0$:
\begin{equation}\label{zerogap}
\!\!\!
\begin{aligned}
&\inf_x\{2g^Tx + x^THx:\;\|x\|^2 \le t\} \\
= & \sup_\lambda \{-g^T(H-\lambda I)^\dagger g + t\lambda:\;\lambda \le 0, H - \lambda I \succeq 0, g\in {\rm Range}(H - \lambda I)\};
\end{aligned}
\end{equation}
note that the above equality holds trivially when $t = 0$, while for $t>0$, the equality follows from the strongly duality of TRS \eqref{Prob_trs}. Define $d(\lambda):= g^T(H-\lambda I)^\dagger g + \delta_D(\lambda)$ with
$D := \{\lambda:\;\lambda \le 0, H - \lambda I \succeq 0, g\in {\rm Range}(H - \lambda I)\}$, then one can check that $d$ is proper closed convex and we see from \eqref{zerogap} that for any $t\ge 0$,
\begin{equation}\label{dconjugate}
d^*(t) = \inf_x\{2g^Tx + x^THx:\;\|x\|^2 \le t\} < \infty.
\end{equation}

Now, using the assumption that $\rho$ is nondecreasing (see
Assumption~\ref{assump:one}), we can rewrite \eqref{Prob_general} as follows:
\begin{equation}\label{derivation}
  \begin{aligned}
    v_{\rho r} & = \inf_{t\ge 0, x} \{2g^Tx + x^THx + \rho(t):\; \|x\|^2 \le t\}\\
    & = \inf_{t\ge 0} \left\{\rho(t) + \inf_{\|x\|^2 \le t}\{2g^Tx + x^THx\}\right\}\\
    & = \inf_{t\ge 0}\left\{\rho(t) + d^*(t)\right\} = \sup_{\lambda \in D}\left\{-g^T(H-\lambda I)^\dagger g - \rho^+(-\lambda)\right\},
  \end{aligned}
\end{equation}
where $\rho^+(u):= \sup_{t\ge 0}\{ut - \rho(t)\}$ is the monotone
conjugate, the third equality follows from \eqref{dconjugate}, and the
last equality holds as a consequence of the Fenchel-duality
theorem~\cite[Theorem~3.3.5]{BoLe:00} because ${\rm dom}\, d^* \supseteq
\R_+$ (thanks to \eqref{dconjugate}) and ${\rm dom}\,(\rho +
\delta_{\R_+})\cap \R_{++}\neq \emptyset$ by Assumption~\ref{assump:one}. The same theorem guarantees that the supremum in \eqref{derivation} is attained when finite. This completes the proof. \qed
\end{proof}

We refer to \eqref{L_dual_gen} as the Lagrange dual problem of $\rho$RS
\eqref{Prob_general}. As is shown in Theorem~\ref{thm0}, strong duality holds between \eqref{Prob_general} and its Lagrange dual problem \eqref{L_dual_gen} under Assumption~\ref{assump:one}.
Our next theorem builds on this result and gives necessary and sufficient optimality conditions for $\rho$RS, which can be nonconvex in general.
\begin{theorem}[Necessary and sufficient conditions for optimality]\label{thm:opt_cond}
  Consider \eqref{Prob_general} and suppose that Assumption~\ref{assump:one} holds. Then $x^*$ is an optimal solution of $\rho$RS \eqref{Prob_general} if and only if there exists $\lambda^*$ such that
  \begin{equation}\label{optcon:gen}
  \begin{aligned}
    &(H - \lambda^*I)x^* = -g,\\
    &H - \lambda^* I \succeq 0,\  \lambda^*\le 0,\\
    &- \lambda^* \in \partial(\rho + \delta_{\R_+})(\|x^*\|^2).
  \end{aligned}
  \end{equation}
\end{theorem}
\begin{proof}
  We start by noting the following chain of inequalities:
  \begin{equation}\label{1st_opt2}
    \begin{aligned}
      v_{\rho r} & = \inf_{t\ge 0, x} \{2g^Tx + x^THx + \rho(t):\; \|x\|^2 \le t\}\\
      & = \inf_{t\ge 0, x}\sup_{\lambda\le 0}\{2g^Tx + x^T(H - \lambda I)x + \lambda t + \rho(t)\}\\
      & \ge \sup_{\lambda\le 0}\inf_{t\ge 0, x}\{2g^Tx + x^T(H - \lambda I)x + \lambda t + \rho(t)\}\\
      & = \sup_{\lambda\le 0}\left\{\inf_x\{2g^Tx + x^T(H - \lambda I)x\} + \inf_{t\ge 0}\{\lambda t + \rho(t)\}\right\}\\
      & = \sup_{\lambda \in D}\left\{-g^T(H-\lambda I)^\dagger g - \rho^+(-\lambda)\right\} = v_{\rho d},
    \end{aligned}
  \end{equation}
  where $D := \{\lambda:\;\lambda \le 0,\; H - \lambda I \succeq 0,\; g\in
{\rm Range}(H - \lambda I)\}$. In view of Theorem~\ref{thm0}, the inequality in \eqref{1st_opt2} must then hold as an equality. Moreover, the supremum in the third and fourth lines are attained at some $\lambda^*\in D$ when $v_{\rho r}$ is finite.

  Now, suppose that $x^*$ solves \eqref{Prob_general}. From the above discussion, we see that 
  \begin{equation}\label{chainineq}
  \begin{aligned}
    v_{\rho r} & = 2g^Tx^* + {x^*}^THx^* + \rho(\|x^*\|^2)\\
    & = \inf_{t\ge 0, x}\{2g^Tx + x^T(H - \lambda^* I)x + \lambda^* t + \rho(t)\}\\
    & \le 2g^Tx^* + {x^*}^T(H - \lambda^* I)x^* + \lambda^* \|x^*\|^2 + \rho(\|x^*\|^2) = v_{\rho r},
  \end{aligned}
  \end{equation}
  where $\lambda^*\in D$ is a maximizer that achieves the supremum in the third line of \eqref{1st_opt2} and the second equality follows from the 3rd relation in \eqref{1st_opt2} (which has been proven to hold as an equality).
  Thus, the infimum in the above display is attained at $(x^*,t^*) = (x^*,\|x^*\|^2)$. We then deduce from the first-order optimality condition that
  \[
  (H - \lambda^* I)x^* = -g,\ -\lambda^*\in \partial(\rho + \delta_{\R_+})(\|x^*\|^2).
  \]
  Moreover, for the infimum in \eqref{chainineq} to be finite when minimizing with respect to $x$, we necessarily have $H - \lambda^* I \succeq 0$.
  This proves \eqref{optcon:gen}.

  Conversely, suppose that $x^*$ is such that there exists $\lambda^*\le 0$ satisfying \eqref{optcon:gen}. Then we have
  \[
  \begin{aligned}
    2g^Tx^* + {x^*}^THx^* + \rho(\|x^*\|^2)
    & \le \inf_{t\ge 0, x}\{2g^Tx + x^T(H - \lambda^* I)x + \lambda^* t + \rho(t)\}\\
    & \le \sup_{\lambda\le 0}\inf_{t\ge 0, x}\{2g^Tx + x^T(H - \lambda I)x + \lambda t + \rho(t)\}\\
    & \le \inf_{t\ge 0, x} \{2g^Tx + x^THx + \rho(t):\; \|x\|^2 \le t\} = v_{\rho r},
  \end{aligned}
  \]
  where: the first inequality holds because the first and third relations in \eqref{optcon:gen} are simply the first-order optimality conditions of the optimization problem $\inf_{t\ge 0, x}\{2g^Tx + x^T(H - \lambda^* I)x + \lambda^* t + \rho(t)\}$, which is convex since $H - \lambda^* I \succeq 0$; the last inequality follows from \eqref{1st_opt2}. This shows that $x^*$ solves \eqref{Prob_general}, as required. This completes the proof. \qed
\end{proof}

\subsection{RW-duality}\label{sec:RW_dual}

In this section, we discuss the RW-dual problem for $\rho$RS \eqref{Prob_general}. Recall that the RW algorithm was originally designed for the following equality constrained TRS:
\begin{equation}\label{Prob_trs=}
  \begin{array}{rl}
    v_{_{{\rm TRS}_=}} := \min\limits_x &  2g^Tx + x^THx\\
    {\rm s.t.}& \|x\|^2 = s,
  \end{array}
\end{equation}
where $H$, $g$ and $s$ are as in \eqref{Prob_trs}.
This algorithm proceeds by solving the RW-dual problem defined below, and then recovering the primal solution:
\begin{equation*}
    \sup_t\ k(t) - t,
\end{equation*}
where
\[
k(t) := \inf_{y\in\R^{n+1}}\left\{y^T\begin{bmatrix}
  t& g^T\\ g&H
\end{bmatrix}y :\; \|y\|^2 = s + 1\right\}.
\]
It can be shown that the function $k$ is concave and continuously
differentiable, except perhaps at one point. Moreover, the function values and
supergradients of $k$ can be obtained by computing the smallest eigenvalue and
a corresponding eigenvector of an $(n+1)\times (n+1)$ matrix. This can
be done efficiently if $H$ is sparse even when $n$ is large, and is the key to the implementation of efficient algorithms for solving the RW-dual. In what follows, we first provide an alternative derivation for the above RW-dual for \eqref{Prob_trs=}, which reveals its intrinsic connection to the standard Lagrange dual problem of \eqref{Prob_trs=}. We then use a similar procedure to derive (and define) the RW-dual problem for $\rho$RS.

\subsubsection{An alternative derivation of the RW-dual for equality constrained TRS}
\label{sect:altderivRWdual}

We start by recalling the Lagrange dual of TRS \eqref{Prob_trs=}.
For the TRS \eqref{Prob_trs=}, its Lagrange dual~\cite{StWo:93} is given by
\begin{equation}\label{L_dual_trs}
  \begin{array}{rl}
    v_{_{{\rm TRS}_=}} = \sup\limits_{\lambda} & -g^T\left( H -\lambda I\right)^\dagger g + s\lambda\\
    {\rm s.t.}& H -\lambda I\succeq 0, g \in {\rm Range}(H -\lambda I).
  \end{array}
\end{equation}
Using Schur's complement, one can rewrite the above problem as follows:
\begin{equation*}
  \begin{aligned}
    v_{_{{\rm TRS}_=}} & = \sup_{\tilde t, \lambda}\left\{-\tilde t + s\lambda:\; \begin{bmatrix}
      \tilde t& g^T\\g& H -\lambda I
    \end{bmatrix}\succeq 0\right\}\\
    & =\sup_{t, \lambda}\left\{-t + (s+1)\lambda:\; \begin{bmatrix}
      t& g^T\\g& H
    \end{bmatrix}-\lambda I\succeq 0\right\}\\
    & = \sup_t\{k(t) - t\},
  \end{aligned}
\end{equation*}
where we used $t = \tilde t + \lambda$ for the second equality, and set
\[
k(t) := \sup_\lambda\left\{(s+1)\lambda:\;\begin{bmatrix}
  t&g^T\\g&H
\end{bmatrix}\succeq \lambda I\right\} = \inf_y\left\{y^T\begin{bmatrix}
  t&g^T\\g&H
\end{bmatrix}y:\; \|y\|^2 = s+1\right\}.
\]
The problem $\sup_t\{k(t)-t\}$ is precisely the RW-dual for the TRS \eqref{Prob_trs=}. Our derivation above connects this dual directly with the Lagrange dual problem \eqref{L_dual_trs}.

\subsubsection{The RW-dual for $\rho$RS}

By mimicking the procedure used in
Section~\ref{sect:altderivRWdual} for deriving the RW-dual of TRS
\eqref{Prob_trs=}, we derive (and define) the RW-dual for $\rho$RS
\eqref{Prob_general}, under Assumptions~\ref{assump:one} and \ref{assump:two}. Specifically, we have the following theorem.

\begin{theorem}[Strong duality for RW-dual]\label{thm:rwdual}
  Consider \eqref{Prob_general} and suppose that Assumptions~\ref{assump:one} and \ref{assump:two} hold. Consider the following concave maximization problem:
  \begin{equation}\label{rw_dual0}
  \ {\rm (RW-dual)}\  \qquad v_{_{{\rm RW}d}} := \sup_t\ \widehat k(t) - t,
  \end{equation}
  where
  \begin{equation}\label{kt10}
  \widehat k(t) = \inf_{\gamma \ge 0}\left\{\rho(\gamma-1) + \gamma \lambda_{\min}\left(\begin{bmatrix}
    t& g^T\\g& H
  \end{bmatrix}\right)\right\}.
  \end{equation}
  Then $v_{_{{\rm RW}d}} = v_{\rho r}$ and a solution of problem \eqref{rw_dual0} exists.
\end{theorem}
\begin{proof}
  Apply Schur's complement to rewrite \eqref{L_dual_gen} as follows:
\begin{equation}\label{deftildet}
  \begin{aligned}
    v_{\rho d} & = \sup_{\lambda\le 0, \tilde t}\left\{-\tilde t - \rho^+(-\lambda):\; \begin{bmatrix}
      \tilde t& g^T\\g& H - \lambda I
    \end{bmatrix}\succeq 0\right\}\\
    & =\sup_{\lambda\le 0, t}\left\{-t + \lambda - \rho^+(-\lambda):\; \begin{bmatrix}
      t& g^T\\g& H
    \end{bmatrix}- \lambda I\succeq 0\right\}\\
    & = \sup_t\{ \widehat h(t) - t\},
  \end{aligned}
\end{equation}
where we let $t:= \tilde t + \lambda$ for the second equality, and set
\begin{equation}\label{kt02}
  \widehat h(t) := \sup_{\lambda\le 0}\left\{\lambda - \rho^+(-\lambda):\; \begin{bmatrix}
      t& g^T\\g& H
    \end{bmatrix}- \lambda I\succeq 0\right\}.
\end{equation}
We next show that $\widehat h = \widehat k$. To this end, we note that
\begin{equation}\label{hchain}
\!\!\!\!\!
\begin{aligned}
  \widehat h(t) & =\sup_{\lambda\le 0}\inf_{U\succeq 0}\left\{\lambda - \rho^+(-\lambda) + \trace \left(U \begin{bmatrix}
    t-\lambda& g^T\\g& H-\lambda I
  \end{bmatrix}\right)\right\}\\
  & \overset{\rm (a)}= \inf_{U\succeq 0}\sup_{\lambda\le 0}\left\{\lambda - \rho^+(-\lambda) + \trace \left(U \begin{bmatrix}
    t-\lambda& g^T\\g& H-\lambda I
  \end{bmatrix}\right)\right\}\\
  & = \inf_{U\succeq 0}\left\{\sup_{\lambda\le 0}\left\{-\lambda(\trace (U)-1) - \rho^+(-\lambda)\right\} + \trace \left(U \begin{bmatrix}
    t& g^T\\g& H
  \end{bmatrix}\right)\right\}\\
  & \overset{\rm(b)}= \inf_{U\succeq 0}\left\{\rho(\trace (U)-1) + \trace \left(U \begin{bmatrix}
    t& g^T\\g& H
  \end{bmatrix}\right)\right\},
\end{aligned}
\end{equation}
where: equality (a) follows from strong duality, because the Slater's condition holds for the optimization problem in \eqref{kt02} and ${\rm dom}\,\rho^+ = \R$ thanks to the supercoercivity assumption in Assumption~\ref{assump:two}; equality (b) follows from the following observation: For any $u\in \R$, it holds that
\[
\begin{aligned}
&\sup_{s\ge 0}\{s u - \rho^+(s)\} = (\rho^++ \delta_{\R_+})^*(u) = ([\rho + \delta_{\R_+}]^*+ \delta_{\R_+})^*(u)\\
& = [\rho + \delta_{\R_+}]\square \delta_{\R_-}(u) = \inf_y \{\rho(y):\; y \ge 0, u\le y\} = \rho(u),
\end{aligned}
\]
where the third equality holds thanks to \cite[Theorem~16.4]{roc70} and the finiteness of $\rho^+ = (\rho +
\delta_{\R_+})^*$, and
the last equality holds because $\rho$ is nondecreasing and $\rho(t) =
0$ whenever $t\le 0$ (see Assumption~\ref{assump:one}).
Looking at $U$ with the same trace in the last relation in \eqref{hchain}, we further have
\[
\begin{aligned}
  \widehat h(t)& = \inf_{\gamma \ge 0}\inf_{U\succeq 0, \trace (U)=\gamma}\left\{\rho(\trace (U)-1) + \trace \left(U \begin{bmatrix}
    t& g^T\\g& H
  \end{bmatrix}\right)\right\}\\
  & = \inf_{\gamma \ge 0}\left\{\rho(\gamma-1) +
\inf_{U\succeq 0, \trace (U)=\gamma}\left\{\trace \left(U \begin{bmatrix}
    t& g^T\\g& H
  \end{bmatrix}\right)\right\}\right\}\\
  & = \inf_{\gamma \ge 0}\left\{\rho(\gamma-1) + \gamma \lambda_{\min}\left(\begin{bmatrix}
    t& g^T\\g& H
  \end{bmatrix}\right)\right\}=\widehat k(t).
\end{aligned}
\]

Now, using the fact that $\widehat h = \widehat k$, we deduce from
\eqref{deftildet} that $v_{\rho d} = v_{_{{\rm RW}d}}$. This together with
Theorem~\ref{thm0} shows that $v_{\rho r} = v_{_{{\rm RW}d}}$. Finally, note that $v_{\rho r}$
is finite thanks to the supercoercivity condition in
Assumption~\ref{assump:two}. Thus,
the Lagrange dual problem \eqref{L_dual_gen} is attained at some
$\lambda^*$ according to Theorem~\ref{thm0}. Using this and \eqref{deftildet}, we deduce that the supremum in \eqref{rw_dual0} is attained at $t^* = g^T(H - \lambda^*I)^\dagger g + \lambda^*$. This completes the proof. \qed
\end{proof}

The expression of $\widehat k$ in \eqref{kt10} will be used in latter sections for discussing how function values and (super)gradients of $\widehat k$ can be computed efficiently.
%

\section{An algorithm for $\rho$RS based on its RW-dual}\label{sec:algorithm}
\subsection{Easy and hard cases}
As a preparation for our algorithmic discussion, we discuss the so-called easy-case and hard-case instances for \eqref{Prob_general}. Our definitions parallel the existing definitions of easy-case and hard-case instances in the literature of TRS; see, for example,~\cite{ConGouToi:00,FortinWolk:03}. We start by looking at alternative characterizations of \eqref{optcon:gen}.

Suppose that Assumption~\ref{assump:one} holds.
We note first that the second relation in \eqref{optcon:gen} means that $ \lambda^*\leq \widetilde \lambda:=\min\{0,\lambda_{\min}(H)\} $, and we have the following equivalent reformulation of the third relation in \eqref{optcon:gen}:
\begin{equation}\label{optcond:eq3}
\|x^*\|^2 \in\partial (\rho^+)(-\lambda^*).
\end{equation}
The next proposition further reformulate \eqref{optcond:eq3} equivalently into an equation, under additional assumptions.

\begin{prop}\label{prop:diffcase}
	Consider \eqref{Prob_general} and suppose that
Assumptions~\ref{assump:one}-\ref{assump:three} hold. Let $\lambda^* < \widetilde \lambda$, where $\widetilde \lambda = \min\{0,\lambda_{\min}(H)\}$. Then $ \lambda^* $ and $ x^* $ satisfy \eqref{optcon:gen} if and only if $ \lambda^* $ satisfies
	\begin{equation}\label{eq:nonleq}
		\|(H-\lambda^* I)^{-1}g\|^2 = (\rho^+)'(-\lambda^*),
	\end{equation}
	and $x^*$ is given by
\begin{equation}\label{xsoln}
x^* = - (H-\lambda^*I)^{-1}g.
\end{equation}
\end{prop}
\begin{proof}
    We first assume that $ x^* $ and $ \lambda^* $ satisfy \eqref{optcon:gen}. As $ \lambda^*<\widetilde \lambda $, it holds that $H-\lambda^* I \succ 0$ and we deduce from the first relation in \eqref{optcon:gen} that \eqref{xsoln} holds.
	Now, from the third relation in \eqref{optcon:gen}, we see that
	\[
	\|x^*\|^2 \in \partial (\rho + \delta_{\R_+})^*(-\lambda^*) = \{(\rho^+)^\prime(-\lambda^*)\},
	\]
	where the inclusion follows from Young's inequality, and the
equality follows from the definition of monotone conjugate,
Assumption~\ref{assump:three} and the fact that $\lambda^*<\widetilde \lambda \le 0$.
	The above display together with \eqref{xsoln} shows that \eqref{eq:nonleq} holds.
	
    Then we prove the converse implication. Suppose that $ \lambda^* $ satisfies \eqref{eq:nonleq} and $ x^* $ is given by \eqref{xsoln}. Using \eqref{xsoln} and the assumption $ \lambda^*<\widetilde \lambda $, we see immediately that $ (H-\lambda^*I)x^* = -g $, $ H\succeq \lambda^* I$ and $ \lambda^*\leq 0 $. Moreover, the relation $-\lambda^*\in\partial(\rho+\delta_{\R_+})(\|x^*\|^2) $ in \eqref{optcon:gen} is a consequence of \eqref{eq:nonleq}, \eqref{xsoln} and Young's inequality. Thus, we have shown that $ x^* $ and $ \lambda^* $ satisfy \eqref{optcon:gen}. \qed
\end{proof}

Notice from the above proposition that, if Assumptions~\ref{assump:one}-\ref{assump:three} hold and $\lambda^* < \widetilde \lambda$, then the $ \lambda^* $ and $
x^* $ satisfying \eqref{optcon:gen} can be obtained by solving the nonlinear equation \eqref{eq:nonleq}.
We present in the next proposition a checkable {\it sufficient condition} for $\lambda^* < \widetilde \lambda  $. Before proceeding, we first prove an auxiliary lemma concerning the monotone conjugate $ \rho^+ $.
\begin{lem}[Properties of $\rho^+$]\label{auxiliary_lem}
	For a proper closed convex function $\rho$ satisfying
Assumptions~\ref{assump:one}-\ref{assump:three}, the following statements hold.
	\begin{enumerate}[{\rm (i)}]
		\item $\rho^+(u) = 0$ whenever $u < 0$.
		\item For any fixed $M > 0$ and any $s\in (0,M)$, it holds that
		\[
		0\le (\rho^+)'(s) \le (\rho^+)'(M).
		\]
		\item There exists $\hat t > 0$ satisfying $(\rho^+)'(\hat t) > 0$.
	\end{enumerate}
\end{lem}
\begin{proof}
	Note from Assumption~\ref{assump:one} that $\rho$ is nondecreasing with $\rho(u)=0$ whenever $u\le 0$. Hence, $\rho$ is nonnegative and $\rho(0)=0$. Using this and the definition of monotone conjugate, we can see that item (i) holds.
	Furthermore, we deduce immediately from item (i) that $(\rho^+)'(u) = 0$ whenever $u < 0$. This together with the convexity of $\rho^+$ proves item (ii).
	
	We now prove item (iii). Suppose to the contrary that $(\rho^+)'(t) \equiv 0$ for all $t > 0$. Then it follows that $\rho^+ \equiv 0$ on $(0,\infty)$. Also, we have from item (i) that $\rho^+\equiv 0$ on $(-\infty,0)$. Thus, the convex function $\rho^+$ must be constantly zero. This implies that
	\[
	\rho + \delta_{\R_+}=(\rho + \delta_{\R_+})^{**} = (\rho^+)^* = \delta_{\{0\}},
	\]
	contradicting Assumption~\ref{assump:one}. This completes the proof. \qed
\end{proof}

\begin{prop}\label{prop:infty}
Consider \eqref{Prob_general} and suppose that
Assumptions~\ref{assump:one}-\ref{assump:three} hold. Then the function $\lambda\mapsto \|(H-\lambda I)^{-1}g\|^2 - (\rho^+)^\prime(-\lambda)$
is continuous and strictly increasing on $ (-\infty, \widetilde \lambda) $, where $\widetilde\lambda = \min\{0,\lambda_{\min}(H)\} $, and it holds that
\begin{equation}\label{valueinfinity}
	\limsup_{\lambda\to-\infty}\ \|(H-\lambda I)^{-1}g\|^2 - (\rho^+)'(-\lambda) < 0.
\end{equation}
If we assume in addition that $ \lim_{\lambda\uparrow \widetilde \lambda} \|(H-\lambda I)^{-1}g\|^2 = \infty $, then there exists a unique $ \lambda^*\in(-\infty, \widetilde\lambda) $ such that \eqref{eq:nonleq} holds.
\end{prop}
\begin{proof}

We note first that $ H-\lambda I\succ 0 $ whenever $\lambda\in (-\infty, \widetilde \lambda) $, and $ g\neq 0 $. Thus, we have
	\[
	(\|(H-\lambda I)^{-1}g\|^2)^\prime = 2g^T(H-\lambda I)^{-3}g >0,\quad  \forall \lambda \in (-\infty, \widetilde \lambda).
	\]
	This means that $ \lambda\mapsto \|(H-\lambda I)^{-1}g\|^2 $ is strictly increasing on $(-\infty, \widetilde \lambda)$.
	Moreover, according to~\cite[Theorem~25.5]{roc70}, Assumption~\ref{assump:three} and the convexity of $\rho^+$, we have that $(\rho^+)'$ is continuous and nondecreasing on $(0,\infty)$. In view of these, we deduce further that $\lambda \mapsto \|(H-\lambda I)^{-1}g\|^2 - (\rho^+)'(-\lambda)$ is continuous and strictly increasing on $(-\infty,\widetilde\lambda)$.

Next, we see from Lemma~\ref{auxiliary_lem} item (iii) and the convexity of $\rho^+$ that there exists $\hat t > 0$ so that $(\rho^+)'(t)\ge (\rho^+)'(\hat t) > 0$ for all $t > \hat t$. A direct computation then gives
\begin{equation*}
	\limsup_{\lambda\to-\infty}\ \|(H-\lambda I)^{-1}g\|^2 - (\rho^+)'(-\lambda) \le \lim_{\lambda\to-\infty}\ \|(H-\lambda I)^{-1}g\|^2 - (\rho^+)'(\hat t) < 0.
\end{equation*}

Finally, suppose in addition that $ \lim_{\lambda\uparrow \widetilde \lambda}\|(H-\lambda I)^{-1}g\|^2 = \infty $. Since $(\rho^+)'$ is continuous on $(0,\infty)$ and we have $\lim_{t\downarrow 0}(\rho^+)'(t) = (\rho^+)_+^\prime(0)$ thanks to Assumption~\ref{assump:three}, the existence of $\lim_{t\downarrow 0}(\rho^+)'(t)$ (see Lemma~\ref{auxiliary_lem} item (ii)) and the L'Hospital's rule, we deduce that
\[
	\lim_{\lambda\uparrow \widetilde \lambda}\  \|(H - \lambda I)^{-1} g\|^2 - (\rho^+)'(-\lambda)=\lim_{\lambda\uparrow \widetilde \lambda}\  \|(H - \lambda I)^{-1} g\|^2 - (\rho^+)_+'(-\widetilde\lambda)  = \infty,
\]
since $\widetilde\lambda\le 0$.
Combining this with \eqref{valueinfinity} and the fact that $\lambda\mapsto \|(H - \lambda I)^{-1} g\|^2 -  (\rho^+)'(-\lambda) $ is continuous and strictly increasing on $(-\infty,\widetilde\lambda)$, we see that there exists a unique $\lambda^* \in (-\infty,\widetilde\lambda)$ so that
\[
\|(H - \lambda^* I)^{-1} g\|^2 =  (\rho^+)'(-\lambda^*).
\]
This completes the proof. \qed
\end{proof}

We are now ready to present our definition of easy and hard cases for \eqref{Prob_general}.
\begin{definition}[{Easy and hard cases for \eqref{Prob_general}}]\label{def:easy_case}
	Consider \eqref{Prob_general} and suppose that
Assumptions~\ref{assump:one}-\ref{assump:three} hold. We say that the easy case occurs if $ \lim_{\lambda\uparrow \widetilde\lambda} \|(H-\lambda I)^{-1}g\|^2 = \infty$, where $\widetilde \lambda = \min\{0,\lambda_{\min}(H)\}$. Otherwise, we have the hard case.
\end{definition}
Thus, in view of Definition~\ref{def:easy_case}, Propositions~\ref{prop:diffcase} and \ref{prop:infty}, when Assumptions~\ref{assump:one}-\ref{assump:three} hold, we see that easy case instances of \eqref{Prob_general} can be solved by first finding $\lambda^*$ as the unique solution of the equation \eqref{eq:nonleq} and then computing the optimal $x^*$ via \eqref{xsoln}. We next deal with the scenarios when $ \lim_{\lambda\uparrow \widetilde \lambda} \|(H-\lambda I)^{-1}g\|^2 < \infty$, i.e., the hard case instances. We start with the following simple alternative characterization of $ \lim_{\lambda\uparrow \widetilde \lambda} \|(H-\lambda I)^{-1}g\|^2 = \infty$. The result should be well known and we include its proof for completeness.
\begin{lem}\label{hehehaha}
	Consider \eqref{Prob_general} and let $\widetilde \lambda = \min\{0,\lambda_{\min}(H)\}$. Then we have $ \lim_{\lambda\uparrow \widetilde \lambda} \|(H-\lambda I)^{-1}g\|^2 = \infty$ if and only if $ \lambda_{\min}(H)\le 0 $ and $ g\notin  {\rm Range}(H-\widetilde \lambda I)$.
\end{lem}
\begin{proof}
	We first prove the sufficiency. Note that $ \widetilde \lambda = \lambda_{\min}(H) $ since $ \lambda_{\min}(H)\le 0 $. Now, let $ H = Q\Lambda Q^T $ be an eigenvalue decomposition of $ H $. Then $ \lambda \mapsto \|(H-\lambda I)^{-1}g\|^2 $ is continuous on $ (-\infty,\lambda_{\min}(H)) $ and
	\[
	\|(H-\lambda I)^{-1}g\|^2 =\sum_{i=1}^{n}\dfrac{(q_i^Tg)^2}{(\lambda_i-\lambda)^2},
	\]
	where $q_i$ is the $ i $-th column of $ Q $, and $ \lambda_i $ is the $ i $-th diagonal element of $ \Lambda $. Since $g\notin {\rm Range}(H-\widetilde \lambda I)$, we must have $q_j^Tg\neq 0$ for some $q_j$ with $ H q_j = \lambda_{\min}(H) q_j $. Consequently, we have
\[
\lim_{\lambda\uparrow \widetilde \lambda} \|(H-\lambda I)^{-1}g\|^2 = \lim_{\lambda\uparrow\lambda_{\min}(H)}\|(H-\lambda I)^{-1}g\|^2 = \infty.
\]
This proves the sufficiency part.
	
We now prove the necessity by proving the contrapositive statement. First, suppose $ \lambda_{\min}(H) > 0 $. It implies that $ \widetilde\lambda = 0 < \lambda_{\min}(H) $. Since $\lambda\mapsto\|(H-\lambda I)^{-1}g\|^2$ is continuous on $ (-\infty,\lambda_{\min}(H)) $, we have $ \lim_{\lambda\uparrow \widetilde \lambda}\|(H-\lambda I)^{-1}g\|^2< \infty $. Next, suppose $ g\in {\rm Range}(H-\widetilde \lambda I) $ and $ \lambda_{\min}(H) \leq 0$. Then $ g\in {\rm Range}(H-\widetilde \lambda I) = {\rm Range}(H - \lambda_{\min}(H) I)$ and hence
  \[
  \lim_{\lambda\uparrow \widetilde \lambda}\|(H-\lambda I)^{-1}g\|^2 = \lim_{\lambda\uparrow \lambda_{\min}(H)}\|(H-\lambda I)^{-1}g\|^2 <\infty.
  \]
	This completes the proof. \qed
\end{proof}

Based on Lemma~\ref{hehehaha}, we can now restate the hard case instances of \eqref{Prob_general} to be instances that either have $H\succ 0$ or $g\in {\rm Range}(H - \widetilde \lambda I)$, where $\widetilde \lambda = \min\{0,\lambda_{\min}(H)\}$. Note that we have $ \widetilde \lambda = 0 $ when $ H\succeq 0 $. It means that we always have $ g\in {\rm Range}(H-\widetilde \lambda I)  $ when $ H\succ 0 $. Thus, the hard case instances can be further characterized by either ``$H\succ 0$", or ``$\lambda_{\min}(H) \le 0$ and $g\in {\rm Range}(H - \lambda_{\min}(H)I)$". We discuss in the next theorem how to deal with these cases.

\begin{theorem}\label{thm:hardcases}
Consider \eqref{Prob_general} and suppose that Assumptions~\ref{assump:one}-\ref{assump:three} hold.
\begin{enumerate}[{\rm (i)}]
	\item Suppose that $H\succ 0$. If $\|H^{-1}g\|^2 \in \argmin \rho$, then $x^* = -H^{-1}g$ solves \eqref{Prob_general}. Otherwise, there exists a unique $\lambda^* < 0$ such that $ \|(H-\lambda^* I)^{-1}g\|^2 = (\rho^+)'(-\lambda^*) $ holds and $ x^* =-(H-\lambda^* I)^{-1}g $ solves \eqref{Prob_general}.
	\item Suppose that $\lambda_{\min}(H) \le 0$ and $g\in {\rm Range}(H - \lambda_{\min}(H)I)$. Then $\lambda \mapsto (H - \lambda I)^\dagger g$ is continuous on $(-\infty,\lambda_{\min}(H)]$.
	
	If $\|(H - \lambda_{\min}(H) I)^\dagger g\| >  \sqrt{(\rho^+)_+'(-\lambda_{\min}(H))}$, then there exists a unique $\lambda^* < \lambda_{\min}(H)$ such that $ \|(H-\lambda^* I)^{-1}g\|^2 = (\rho^+)'(-\lambda^*) $ holds and $ x^* =-(H-\lambda^* I)^{-1}g $ solves \eqref{Prob_general}.
		
	On the other hand, if $\|(H - \lambda_{\min}(H) I)^\dagger g\| \le  \sqrt{(\rho^+)_+'(-\lambda_{\min}(H))}$, then a solution of \eqref{Prob_general} is given by
	\begin{equation}\label{xhardcase}
		x^* = -(H - \lambda_{\min}(H) I)^\dagger g + \alpha^* v^*	
	\end{equation}
	for some $\alpha^* \in \R$ and $v^*\in \ker(H - \lambda_{\min}(H)I )\backslash\{0\}$ satisfying
	\[
	(\rho^+)_+'(-\lambda_{\min}(H)) = \|(H - \lambda_{\min}(H) I)^\dagger g - \alpha^* v^*\|^2.
	\]
\end{enumerate}
\end{theorem}
\begin{proof}
We first prove item (i).
Since $ H\succ 0 $ and $g\neq 0$, we have $ \|H^{-1}g\|^2>0 $. Thus, if $ \|H^{-1}g\|^2\in\argmin\rho $ holds, then
\begin{equation}\label{aaa}
 0\in\partial (\rho+\delta_{\R_+})(\|H^{-1}g\|^2).
\end{equation}
Hence, we see that \eqref{optcon:gen} is satisfied with $ \lambda^*=0 $ and $ x^*=-H^{-1}g $, showing that $x^*$ solves \eqref{Prob_general}.
	
On the other hand, suppose that $ \|H^{-1}g\|^2\notin\argmin\rho $. Note
that this is the same as $ 0\notin\partial
(\rho+\delta_{\R_+})(\|H^{-1}g\|^2) $, which is further equivalent to $
\|H^{-1}g\|^2 \notin \partial \rho^+(0)$ by Young's inequality. Since
$\partial \rho^+(0) = [(\rho^+)'_-(0),(\rho^+)'_+(0)]$ thanks
to~\cite[Page~229, Line~5]{roc70} and $(\rho^+)'_-(0)=0$ (which follows
immediately from Lemma~\ref{auxiliary_lem} item (i)), we see that $ \|H^{-1}g\|^2 \notin \partial \rho^+(0)$ implies that
\begin{equation}\label{valueat0}
	\|H^{-1}g\|^2 > (\rho^+)_+^\prime(0).
\end{equation}
Moreover, we have $\lim_{t\downarrow 0}(\rho^+)'(t) = (\rho^+)_+^\prime(0)$ thanks to Assumption~\ref{assump:three}, the existence of $\lim_{t\downarrow 0}(\rho^+)'(t)$ (see Lemma~\ref{auxiliary_lem} item (ii)) and the L'Hospital's rule. Using these together with \eqref{valueinfinity} and \eqref{valueat0}, and applying the intermediate value theorem, we see that there exists a $ \lambda^*\in(-\infty, 0) $ such that
\[
	\|(H - \lambda^* I)^{-1} g\|^2 =  (\rho^+)^\prime(-\lambda^*).
\]
The uniqueness of $\lambda^*$ follows from Proposition~\ref{prop:infty}, which asserts that the function $\lambda\mapsto \|(H-\lambda I)^{-1}g\|^2 - (\rho^+)^\prime(-\lambda)$
is continuous and strictly increasing on $ (-\infty, \widetilde \lambda) $.
Then with the $ \lambda^* $ above, we can get the desired $ x^* $
according to Proposition~\ref{prop:diffcase}.
	
We now prove item (ii). The continuity of $\lambda \mapsto (H - \lambda I)^\dagger g$ on $(-\infty,\lambda_{\min}(H)]$ follows immediately from $g\in {\rm Range}(H - \lambda_{\min}(H)I)$.
Thus, if it holds that $\|(H - \lambda_{\min}(H) I)^\dagger g\|^2 >  (\rho^+)'_+(-\lambda_{\min}(H))$, then we can deduce using \eqref{valueinfinity} that there exists a $\lambda^* \in (-\infty,\lambda_{\min}(H))$ so that
\[
\|(H - \lambda^* I)^{-1} g\|^2 =  (\rho^+)'_+(-\lambda^*).
\]
The uniqueness of $\lambda^*$ again follows from Proposition~\ref{prop:infty}, and the desired conclusion concerning $x^*$ follows from Proposition~\ref{prop:diffcase}.

Finally, suppose that $\|(H - \lambda_{\min}(H) I)^\dagger g\|^2 \le  (\rho^+)'_+(-\lambda_{\min}(H))$ and pick any $v^*\in \ker(H - \lambda_{\min}(H) I)\backslash\{0\}$. Then the function
\[
w(\alpha) := \|(H - \lambda_{\min}(H) I)^\dagger g - \alpha v^*\|^2
\]
is continuous, with $w(0) \le (\rho^+)'_+(-\lambda_{\min}(H))$ and
$\lim_{\alpha\to \infty}w(\alpha) = \infty$. Thus, there exists
$\alpha^*$ so that $w(\alpha^*) = (\rho^+)'_+(-\lambda_{\min}(H))$. Let $x^*:= \alpha^*
v^*-(H - \lambda_{\min}(H) I)^\dagger g$. Then $ x^* $ satisfies $\left(H -\lambda_{\min}(H)I\right)x^* = -g$ and
$-\lambda_{\min}(H)\in \partial(\rho+\delta_{\R_+})(\|x^*\|^2)$. Thus, $x^*$ solves \eqref{Prob_general} according to Theorem~\ref{thm:opt_cond}. This completes the proof. \qed
\end{proof}

To conclude this subsection, we summarize the easy and hard cases for $ \rho $RS \eqref{Prob_general} in Table \ref{Table:easyhard}. For the hard cases where $g \in {\rm Range}(H-\widetilde \lambda I)$:
The hard case 1 in Table \ref{Table:easyhard} corresponds to the cases ``$ \lambda_{\min}(H)\le 0 $ and $\|(H-\lambda_{\min}(H) I)^\dagger g\|^2 > (\rho^+)'_+(-\lambda_{\min}(H))$" and ``$ H\succ 0 $ and $ \|H^{-1}g\|^2\notin\argmin \rho $" ({\it this latter relation is equivalent to $ \|(H-\widetilde \lambda I)^\dagger g\|^2 > (\rho^+)'_+(-\widetilde \lambda) $, thanks to $ \widetilde \lambda = 0 $ and the derivation of \eqref{valueat0}}); the hard case 2 in Table \ref{Table:easyhard} corresponds to the cases ``$ \lambda_{\min}(H)\le 0 $ and $ \|(H-\lambda_{\min}(H) I)^\dagger g\|^2 \le (\rho^+)'_+(-\lambda_{\min}(H)) $" and ``$ H\succ 0 $ and $ \|H^{-1}g\|^2\in\argmin \rho $" ({\it the latter relation is equivalent to $ \|H^{-1}g\|^2\le (\rho^+)'_+(0)$ in view of \eqref{aaa}}). According to Theorem~\ref{thm:hardcases}, explicit solutions can be obtained in hard case 2.
\renewcommand{\arraystretch}{1.5}
\begin{table}[H]
	\caption{Easy and hard cases for $\rho$RS }\label{Table:easyhard}
	\begin{threeparttable}
		\centering
		{\small
			\begin{tabular}{| l | l | l |}\hline
				Easy case &   Hard case 1 & Hard case 2 \\  \hline
				$g \notin {\rm Range}(H-\widetilde \lambda I)$ &  $g \in {\rm Range}(H-\widetilde \lambda I)$  & $g \in {\rm Range}(H-\widetilde \lambda I)$ \\
				(implies $\lambda^* < \widetilde \lambda $, thus $\lambda^*<0$)  & $\lambda^*<\widetilde \lambda$ (thus $\lambda^*<0 $) &$\lambda^*=\widetilde \lambda\le 0$  \\
				~ & $ \|(H-\widetilde \lambda I)^{\dagger}g\|^2> (\rho^+)'_+(-\widetilde \lambda) $ &  $ \|(H-\widetilde \lambda I)^\dagger g\|^2 \le (\rho^+)'_+(-\widetilde \lambda) $  \\ \hline
			\end{tabular}
			\begin{tablenotes}
				\item[*] Here $ \widetilde \lambda = \min\{0, \lambda_{\min}(H)\}. $
			\end{tablenotes}
		}
	\end{threeparttable}
\end{table}

\subsection{Properties of the dual function $\widehat k$}\label{sec:propk}

We collect some properties of $\widehat k$ in \eqref{kt10} that will be useful in our subsequent algorithmic development. For notational simplicity, from now on, we write
\begin{equation}\label{define}
\begin{aligned}
\lambda(t)&:= \lambda_{\min}\left(\begin{bmatrix}
    t& g^T\\g& H
  \end{bmatrix}\right)\\
  {\rm and}\ \ \ \bar t := \lim\limits_{\lambda \uparrow \lambda_{\min}(H)}&[\lambda + g^T(H - \lambda I)^{-1} g]\in (-\infty,\infty].
\end{aligned}
\end{equation}
We first recall the following lemma concerning $\lambda(t)$, which is a direct consequence of the inverse function theorem. The results can be found in \cite[Theorem~7.1]{FortinWolk:03} and \cite[Corollary~7.2]{FortinWolk:03}.

\begin{lem}[Properties of $\lambda(t)$]\label{lem1}
  Let $\lambda(t)$ and $\bar t$ be defined as in \eqref{define}. Then the function $t\mapsto \lambda(t)$ is continuous and well-defined everywhere. It is differentiable on $(-\infty,\bar t)\cup (\bar t,\infty)$. Moreover, the following statements hold:
  \begin{enumerate}[{\rm (i)}]
    \item When $t\in (\bar t,\infty)$, it holds that $\lambda(t) = \lambda_{\min}(H)$.
    \item When $t\in (-\infty,\bar t)$, we have that $\lambda(t)$ is the unique root
of the equation
  \[
  t = \lambda + g^T(H - \lambda I)^{-1} g
  \]
  on $(-\infty,\lambda_{\min}(H))$, and, in addition, $\lambda'(t)>0$.
  \end{enumerate}
\end{lem}

\begin{theorem}[Properties of $\widehat k$]\label{thm1}
Consider \eqref{Prob_general} and suppose that
Assumptions~\ref{assump:one}-\ref{assump:three}
hold. Let $\widehat k$ be given in \eqref{kt10}, and let $\lambda(t)$ and $\bar
t$ be defined as in \eqref{define}. Then the following statements hold.
  \begin{enumerate}[{\rm (i)}]
    \item For each $t\in \R$, the infimum in \eqref{kt10} is attained, with the set of minimizers given by
    \[
    \begin{cases}
      \{0\} & {\rm if}\ \lambda(t)> 0,\\
      (\argmin\rho + 1)\cap \R_+ & {\rm if}\ \lambda(t)= 0,\\
      \left\{(\rho^+)'(-\lambda(t))+1\right\} & {\rm otherwise}.
    \end{cases}
    \]
    \item $\widehat k$ is a continuous concave function.
    \item The set of superdifferential of $\widehat k$ at any $t\in \R$
is the convex hull of the set of all numbers of the form of $\gamma v_0^2$, where
    $\gamma$ attains the infimum in \eqref{kt10}, and $\begin{bmatrix}
    v_0& \bar v^T
  \end{bmatrix}^T\in \R^{n+1}$ is a unit eigenvector of $\begin{bmatrix}
    t& g^T\\g& H
  \end{bmatrix}$ corresponding to $\lambda(t)$.
    \item $\widehat k$ is differentiable except possibly at the two points:
\[
\bar t; \qquad \ {\rm and}\ {\rm also}\
g^TH^{-1}g \ {\rm if}\  \lim_{t\uparrow \bar t}\lambda(t) >0.
\]
  \end{enumerate}
\end{theorem}

\begin{proof}
  \begin{enumerate}[{\rm (i)}]
    \item
\label{item:itheorem}
For simplicity of notation, we write $\widehat\lambda:= \lambda(t)$.
Now, define the convex function
\[
f(\gamma) := \rho(\gamma - 1) + \widehat
\lambda\gamma + \delta_{\R_+}(\gamma).
\]

Recall that $\rho$ is nondecreasing thanks to Assumption~\ref{assump:one}.
Thus, if $\widehat \lambda > 0$, then necessarily $\argmin f = \{0\}$. On the other hand, if $\widehat \lambda = 0$, then $f$ is minimized when $\rho(\gamma - 1) = 0$ and $\gamma \ge 0$, which translates to $\argmin f = (\argmin\rho + 1)\cap \R_+$.

Finally, suppose that $\widehat \lambda < 0$. Note that $\argmin f$ is nonempty thanks to the supercoercivity assumption on $\rho$. Moreover, since $\rho(\gamma - 1) = 0$ whenever $\gamma \le 1$ according to Assumption~\ref{assump:one} and $\widehat \lambda < 0$, we must have $\gamma^*\ge 1$ for any $\gamma^*\in \argmin f$. Thus, any $\gamma^*\in \argmin f$ also minimizes $f$ on the interval $[1,\infty)$, i.e., $0 \in \partial(f + \delta_{[1,\infty)})(\gamma^*)$. We then have
\[
-\widehat \lambda \in \partial(\rho + \delta_{\R_+})(\gamma^*-1),
\]
which implies
\[
\gamma^*-1 \in \partial(\rho + \delta_{\R_+})^*(-\widehat \lambda) = \{(\rho^+)'(-\widehat\lambda)\},
\]
where the inclusion follows from Young's inequality, and the equality follows from the definition of monotone conjugate, Assumption~\ref{assump:three} and the fact that $-\widehat \lambda > 0$. This proves item (\ref{item:itheorem}).

    \item The concavity follows because $\widehat k$ is the infimum of concave functions ($\gamma \ge 0$)
    \[
    t\mapsto \rho(\gamma-1) + \gamma \lambda_{\min}\left(\begin{bmatrix}
    t& g^T\\g& H
  \end{bmatrix}\right).
    \]
    Moreover, the function $\widehat k$ is finite everywhere thanks to item (\ref{item:itheorem}). Thus, $\widehat k$ is continuous.
    \item
\label{item:thmiii} Note that the set of superdifferentials of $\widehat k$ at $t$ equals $-\partial(-\widehat k)(t)$, where $-\widehat k$ is the convex function given by
    \[
    -\widehat k(t) = \sup_{\gamma \ge 0}\left\{-\rho(\gamma-1) + \gamma \lambda_{\max}\left(\begin{bmatrix}
    -t& -g^T\\-g& -H
  \end{bmatrix}\right)\right\}.
    \]

    Moreover, notice from Assumptions~\ref{assump:one} and \ref{assump:two} that $\argmin \rho$ is
nonempty and that $\argmin \rho \cap [-M,\infty)$ is bounded for any $M
\ge 0$. This together with Lemma~\ref{auxiliary_lem} item (ii), the continuity of $t\mapsto \lambda(t)$ and item (\ref{item:itheorem}) shows that for any $\delta > 0$, there exists $R > 0$ so that
    \[
    -\widehat k(t) = \sup_{\gamma \in [0,R]}\left\{-\rho(\gamma-1) + \gamma \lambda_{\max}\left(\begin{bmatrix}
    -t& -g^T\\-g& -H
    \end{bmatrix}\right)\right\},
    \]
    whenever $|t| < \delta$; here, $R$ is chosen such that $ [0, R] $ contains all minimizers attaining the infimum in \eqref{kt10} for $|t|< \delta$.
    The desired conclusion now follows from \cite[Theorem~4.4.2, Page~189]{HL}.
    \item
\label{item:theoremiv}
From Lemma~\ref{lem1}, we know that $t\mapsto \lambda(t)$ is differentiable except at $\bar t$. Moreover, one can compute that $\lambda'(t) = 0$ when $t> \bar t$. Thus, if $t>\bar t$ and $\begin{bmatrix}
    v_0& \bar v^T
  \end{bmatrix}^T\in \R^{n+1}$ is a unit eigenvector of $\begin{bmatrix}
    t& g^T\\g& H
  \end{bmatrix}$ corresponding to $\lambda(t)$, then we must have $v_0^2
= 0$. This together with item (\ref{item:thmiii})
shows that $\widehat k$ is differentiable
whenever $t> \bar t$, with $\widehat k'(t) = 0$.

    Next, consider those $t \in (-\infty,\bar t)$. For these $t$,
according to Lemma~\ref{lem1}, $t\mapsto \lambda(t)$ is strictly
increasing and differentiable. Now, if $\lim_{t\uparrow \bar t}\lambda(t) > 0$, then
$\lambda^{-1}(0) = \{g^TH^{-1}g\}$ is a singleton set, and the set of minimizers attaining the infimum in \eqref{kt10} (with $t = g^TH^{-1}g$) is $ (\argmin \rho + 1)\cap \R_+ $ according to item (\ref{item:itheorem}). Hence, we conclude
from items (\ref{item:itheorem}) and (\ref{item:thmiii}) that $\widehat k$
is differentiable on the set difference
$(-\infty,\bar t)\backslash \{g^TH^{-1}g\}$ if $\lim_{t\uparrow \bar t}\lambda(t) > 0$, and is differentiable on $(-\infty,\bar t)$ otherwise.

Thus, $\widehat k$ is differentiable everywhere except at $\bar t$ and, if $\lim_{t\uparrow \bar t}\lambda(t)>0$, also at $g^TH^{-1}g$.\qed
%
%
\end{enumerate}
\end{proof}

\begin{remark}\label{rem3.3}
Note that $\lim_{t\uparrow \bar t}\lambda(t) = \lambda_{\min}(H)$. Thus,
according to Theorem~\ref{thm1} item \eqref{item:theoremiv}, if $H$ is not
positive definite, then $\widehat k$ is differentiable except possibly at one
point: $\bar t$.
\end{remark}

\subsection{Primal solution recovery from RW-dual}

We discuss how one can recover an optimal solution of
\eqref{Prob_general} after solving \eqref{rw_dual0}: recall that
solutions of \eqref{rw_dual0} exist under Assumptions~\ref{assump:one} and \ref{assump:two}, according to Theorem~\ref{thm:rwdual}. We also get a bound on the location of the maximizers of \eqref{rw_dual0}.

\begin{theorem}[Differentiable maximizer of $\widehat k$]\label{thm2}
  Consider \eqref{Prob_general} and suppose that Assumptions~\ref{assump:one}-\ref{assump:three} hold. Let $\lambda(t)$ and $\bar t$ be defined as in \eqref{define} and let $t^*$ be a maximizer of \eqref{rw_dual0}. Then $\lambda(t^*) \le 0$. Suppose in addition that $\widehat k$ in \eqref{kt10} is differentiable at $t^*$. Then $\lambda(t^*) < 0$, and for any unit eigenvector $\begin{bmatrix}
    v_0&\bar v^T
  \end{bmatrix}^T\in \R^{n+1}$ (with $v_0\in \R$) of $\begin{bmatrix}
    t^*& g^T\\g& H
  \end{bmatrix}$ corresponding to $\lambda(t^*)$, it holds that $v_0\neq 0$ and $x^*:= \frac1{v_0}\bar v$ solves \eqref{Prob_general}. Furthermore, the quantities
  \begin{align}
    \tau &:= \!\sup\{\alpha\ge 0:\; \rho(\alpha^2) \le 2\max\{2\|g\|\alpha,\|H\|\alpha^2\}\}, \ {\rm and} \label{tau}\\
    \kappa &:= \!\sup\!\big\{\!\{0\}\!\cup\! \big\{\beta \!>\! \max\{\lambda_{\min}(H),0\}\!:\,\! \beta\sqrt{(\rho^+)'(\beta - \lambda_{\min}(H))} \!\le\! \|g\|\big\}\!\big\} \label{kappa}
  \end{align}
  are finite and we have
  \begin{equation}\label{t_bd0}
  \lambda_{\min}(H) - \zeta \le t^* \le \min\left\{\bar t,\lambda_{\min}(H) + \eta\|g\|\right\}
  \end{equation}
  for any $\eta \ge \tau $ and any $\zeta \ge \kappa$.
\end{theorem}

\begin{proof}
  If $\lambda(t^*) > 0$, we have from Theorem~\ref{thm1}
items \eqref{item:itheorem} and \eqref{item:thmiii}
that $\widehat k^\prime(t^*)= 0$, which means that the objective in \eqref{rw_dual0} has a slope of $-1$ at $t^*$, contradicting the optimality of $t^*$. Thus, $\lambda(t^*) \le 0$.

  We now assume in addition that $\widehat k$ is differentiable at $t^*$.
If $\lambda(t^*) = 0$, then we have from Theorem~\ref{thm1} item \eqref{item:itheorem} that the
infimum in \eqref{kt10} (with $t = t^*$) is attained for any $\gamma\in
(\argmin \rho + 1)\cap \R_+$. Moreover, we see from Assumption~\ref{assump:one} that
$[0,1]\subseteq (\argmin \rho + 1)\cap \R_+$. These together with Theorem~\ref{thm1} item \eqref{item:thmiii} and the
differentiability of $\widehat k$ at $t^*$ implies that $\widehat k^\prime(t^*) = 0$, again contradicting the optimality of $t^*$. Thus, we must have $\lambda(t^*) < 0$.

  Next, let $\begin{bmatrix}
    v_0& \bar v^T
  \end{bmatrix}^T\in \R^{n+1}$ (with $v_0\in \R$) be a unit eigenvector of $\begin{bmatrix}
    t^*& g^T\\g& H
  \end{bmatrix}$ corresponding to $\lambda(t^*)$ and write $\gamma^*:=
(\rho^+)'(-\lambda(t^*))+1$, which is well defined because of Assumption~\ref{assump:three} and $\lambda(t^*)
< 0$. Then, by Theorem~\ref{thm1} item \eqref{item:itheorem}, the infimum in \eqref{kt10} (with $t = t^*$) is uniquely attained by $\gamma^*$. Moreover, Theorem~\ref{thm1} item \eqref{item:thmiii} shows that $\gamma^*v_0^2$ is a supergradient of $\widehat k$ at $t^*$. Since $t^*$ is a maximizer and $\widehat k$ is differentiable at $t^*$, we must then have
  \begin{equation}\label{rel1}
    \gamma^*v_0^2=1,
  \end{equation}
  which further implies that $v_0\neq 0$.

  We now show that $x^*:= \frac1{v_0}\bar v$ solves \eqref{Prob_general}. To
this end, we first note from the definition of an eigenvector that
  \begin{equation}\label{rel2}
  \begin{bmatrix}
    t^* & g^T\\
    g& H
  \end{bmatrix}\begin{bmatrix}
    v_0\\ \bar v
  \end{bmatrix} = \lambda(t^*)\begin{bmatrix}
    v_0\\ \bar v
  \end{bmatrix}.
  \end{equation}
  Dividing both sides of the above relation by $v_0$ and rearranging terms for the second row, we obtain further that
  \begin{equation}\label{opt1}
    (H - \lambda(t^*)I)x^* = -g.
  \end{equation}
  On the other hand, since $\begin{bmatrix}
    v_0& \bar v^T
  \end{bmatrix}^T$ is a unit vector, we have, using the definition of $x^*$, that
  \begin{equation}\label{opt20}
  \|x^*\|^2 = \frac{\|\bar v\|^2}{v_0^2} = \frac{\|\bar v\|^2 + v_0^2}{v_0^2} - 1=  \frac1{v_0^2}-1 = \gamma^* - 1 = (\rho^+)'(-\lambda(t^*)),
  \end{equation}
  where the second last equality follows from \eqref{rel1} and the last
equality follows from the definition of $\gamma^*$. Thus,
  \begin{equation}\label{opt2}
  -\lambda(t^*)\in \partial (\rho + \delta_{\R_+})(\|x^*\|^2).
  \end{equation}
  On the other hand, we have $\lambda(t^*)\le \lambda_{\min}(H)$ by
  the interlacing property of eigenvalues, and hence $H \succeq \lambda(t^*)I$.
  In view of this, $\lambda(t^*) < 0$, \eqref{opt1} and \eqref{opt2}, we
conclude from Theorem~\ref{thm:opt_cond} that $x^*$ is a minimizer of \eqref{Prob_general}.

  Next, note that $\tau = \sup\{\alpha\ge 0:\; \rho(\alpha^2) \le
2\max\{2\|g\|\alpha,\|H\|\alpha^2\}\}\in [0,\infty)$ thanks to the
supercoercivity assumption on $\rho$ in Assumptions~\ref{assump:two}. Also, recall from
Lemma~\ref{auxiliary_lem} item (iii) that $(\rho^+)'(\hat t) > 0$ for some $\hat t > 0$. Thus, $(\rho^+)'(t)\ge (\rho^+)'(\hat t) > 0$ for all $t\ge \hat t$ due to convexity. We can then show via a standard contradiction argument that $\kappa$ is finite.

  Finally, we derive the bound \eqref{t_bd0} for $t^*$. Since
$\lambda(t)=\lambda_{\min}(H)$ whenever $t > \bar t$ according to
Lemma~\ref{lem1}, we conclude using Theorem~\ref{thm1} item~\eqref{item:thmiii} that $\widehat k'(t) = 0$ whenever $t> \bar t$. Hence the interval $(\bar t,\infty)$ cannot contain a maximizer of \eqref{rw_dual0}. Thus, $t^*$ is bounded above by $\bar t$. Next, dividing both sides of \eqref{rel2} by $v_0$ and writing $\lambda^* = \lambda(t^*)$ for notational simplicity, we obtain
  \begin{equation}\label{1st_opt0}
  \begin{aligned}
    (H-\lambda^*I)x^*  &= -g,\\
    t^* + g^Tx^*  &= \lambda^*.
  \end{aligned}
  \end{equation}
  Thus, we have
  \begin{equation}\label{t_upper0}
  t^* =\lambda^* - g^Tx^* \le \lambda_{\min}(H) - g^Tx^* \le \lambda_{\min}(H) + \|g\|\|x^*\|,
  \end{equation}
  where the first inequality follows from interlacing. Since $x^*$ is a minimizer of \eqref{Prob_general} and $\rho(0) = 0$, we have
  \[
  \begin{aligned}
    \rho(\|x^*\|^2)&\le -2g^Tx^* - {x^*}^THx^*\le 2\|g\|\|x^*\| + \|H\|\|x^*\|^2\\
    & \le 2\max\{2\|g\|\|x^*\|, \|H\|\|x^*\|^2\}.
  \end{aligned}
  \]
  This implies that
  \[
  \|x^*\|\le \tau.
  \]
  The second inequality in \eqref{t_bd0} now follows immediately from this and \eqref{t_upper0}.

  It now remains to establish the first inequality in \eqref{t_bd0}. To this end, let $\delta := \lambda_{\min}(H) - \lambda^*$. Then $\delta \ge 0$ due to interlacing, and $H - \lambda^* I \succeq \delta I$. Using these and the two relations in \eqref{1st_opt0}, we obtain that
  \begin{equation}\label{t_rel0}
  \begin{aligned}
  t^* = \lambda^* - g^Tx^* &= \lambda^* + {x^*}^T(H-\lambda^*I)x^*\\
  & \ge \lambda^* + \delta\|x^*\|^2 \ge \lambda^* = \lambda_{\min}(H) - \delta.
  \end{aligned}
  \end{equation}
  If $\delta = 0$, then the first inequality in \eqref{t_bd0} follows immediately from \eqref{t_rel0}. We thus focus on the case when $\delta > 0$. In this case, we have $0\prec (H - \lambda^*I)^{-1}\preceq \delta^{-1}I$ and hence
  \[
  \delta\|x^*\|^2 \le t^* - \lambda^* = g^T (H - \lambda^* I)^{-1}g \le \frac{1}{\delta}\|g\|^2,
  \]
  where the first inequality follows from the first inequality in
\eqref{t_rel0}, and the equality follows from Lemma~\ref{lem1} because $\lambda_{\min}(H) - \lambda^* = \delta > 0$. Consequently,
  \[
  \|g\|\ge \delta \|x^*\| = \delta \sqrt{(\rho^+)'(\delta - \lambda_{\min}(H))},
  \]
  where we made use of \eqref{opt20} and the definition of $\delta$ in the equality. This implies that $\delta \le \kappa$.
  The first inequality in \eqref{t_bd0} now follows immediately from this and \eqref{t_rel0}. This completes the proof. \qed
\end{proof}

We next discuss what happens when $\widehat k$ in \eqref{rw_dual0} is not differentiable at
a maximizer $t^*$ of \eqref{rw_dual0}. Recall from Theorem~\ref{thm1} item \eqref{item:theoremiv} that this can only possibly happen if ``$t^* < \bar t$ and $\lambda(t^*) = 0$" (note that $\lambda(g^TH^{-1}g) = 0$ in this case) or ``$t^* = \bar t$", where $\lambda(t)$ and $\bar t$ are defined in \eqref{define}.

\begin{theorem}[Possibly nondifferentiable maximizer of $\widehat k$]\label{thm3}
Consider \eqref{Prob_general} and suppose that
Assumptions~\ref{assump:one}-\ref{assump:three} hold.
Let $\lambda(t)$ and $\bar t$ be defined as in \eqref{define} and let $t^*$ be a maximizer of \eqref{rw_dual0}. Then the following statements hold.
  \begin{enumerate}[{\rm (i)}]
    \item If $\lambda(t^*) = 0$, then $\lambda_{\min}(H) \ge 0$ and $g\in {\rm Range}(H)$.
    \item If $\lambda(t^*)\neq 0$ and $t^* = \bar t$, then $\lambda(t^*) < 0$, $\lambda(t^*) = \lambda_{\min}(H)$ and $g\in {\rm Range}(H - \lambda_{\min}(H)I)$.
  \end{enumerate}
\end{theorem}
\begin{proof}
  \begin{enumerate}[{\rm (i)}]
    \item Suppose that $\lambda(t^*) = 0$. Since we have from interlacing that $\lambda(t^*)\le \lambda_{\min}(H)$, it follows that $\lambda_{\min}(H) \ge 0$. Then, using the definition of $\lambda(t^*)$ as minimum eigenvalue, we have
        \[
        \begin{bmatrix}
          t^* & g^T\\ g& H
        \end{bmatrix} =
        \begin{bmatrix}
          t^* & g^T\\ g& H
        \end{bmatrix} - \lambda(t^*)I \succeq 0.
        \]
        This implies $g\in {\rm Range}(H)$ as desired.
    \item Suppose that $\lambda(t^*)\neq 0$ and $t^* = \bar t$. We
conclude from Theorem~\ref{thm2} that $\lambda(t^*) < 0$. Moreover, from Lemma~\ref{lem1}, we have by continuity that $\lambda(\bar t) = \lambda_{\min}(H)$. Hence, $\lambda(t^*) = \lambda_{\min}(H)$ and we have
        \[
        \begin{bmatrix}
          t^* & g^T\\ g& H
        \end{bmatrix} - \lambda_{\min}(H)I =
        \begin{bmatrix}
          t^* & g^T\\ g& H
        \end{bmatrix} - \lambda(t^*)I \succeq 0,
        \]
        we conclude further that $g\in {\rm Range}(H - \lambda_{\min}(H)I)$. This completes the proof.\qed
  \end{enumerate}
\end{proof}

From Theorem~\ref{thm3}, we know that if $\widehat k$ in \eqref{rw_dual0} fails
to be differentiable at a maximizer $t^*$ of \eqref{rw_dual0}, then
either ``$H\succ 0$" or ``$g\in {\rm Range}(H - \lambda_{\min}(H)I)$ and $\lambda_{\min}(H)\le 0$". These two
conditions coincide with those of hard-case instances discussed in Theorem~\ref{thm:hardcases}. For hard case 2, we have explicit solutions according to Theorem~\ref{thm:hardcases}. In the next theorem, we show that $\widehat k$ is still differentiable at its maximizer under conditions corresponding to those of hard case 1 in Table \ref{Table:easyhard}; see also the alternative descriptions of hard case 1 preceding that table. Thus, Theorem~\ref{thm2} can again be invoked for obtaining a solution of \eqref{Prob_general} in hard case 1.

\begin{theorem}[Differentiability of $\widehat k$ in hard case 1]\label{thm4}
Consider \eqref{Prob_general} and suppose that
Assumptions~\ref{assump:one}-\ref{assump:three} hold. Let $\lambda(t)$ be defined as in \eqref{define}.
   \begin{enumerate}[{\rm (i)}]
     \item Suppose that $H\succ 0$ and $\|H^{-1}g\|^2 \notin \argmin \rho$. Then it holds that $\lambda(t^*) < 0$, where $t^*$ is a maximizer of \eqref{rw_dual0}. Moreover, $\widehat k$ is differentiable at $t^*$.
     \item Suppose that $\lambda_{\min}(H) \le 0$, $g\in {\rm Range}(H - \lambda_{\min}(H)I)$ and
	 \[
\|(H - \lambda_{\min}(H) I)^\dagger g\| >  \sqrt{(\rho^+)_+'(-\lambda_{\min}(H))}.
\]
Then $\lambda(t^*) < \lambda_{\min}(H)$, where $t^*$ is a maximizer of \eqref{rw_dual0}. Moreover, $\widehat k$ is differentiable at $t^*$.
   \end{enumerate}
\end{theorem}

\begin{proof}
We first prove item (i). Let $ \lambda^* $ be given by  the latter part of
Theorem~\ref{thm:hardcases} item (i). That means $ \lambda^* $ satisfies $ \lambda^*<0<\lambda_{\min}(H) $ and $ \|(H-\lambda^* I)^{-1}g\|^2 = (\rho^+)'(-\lambda^*) $.
Comparing the latter relation with the optimality condition of
\eqref{L_dual_gen}, we see that $\lambda^*$ is a maximizer of
\eqref{L_dual_gen}. Using this and the definition of $\tilde t$ in
\eqref{deftildet}, we deduce further that $t^* = g^T(H - \lambda^*
I)^{-1}g + \lambda^*$ is a maximizer of \eqref{rw_dual0}. We then see
from Lemma~\ref{lem1} that $\lambda(t^*) = \lambda^* < 0
<\lambda_{\min}(H)$ as desired. The claim on differentiability follows
from this and Theorem~\ref{thm1} item (iv).

We now prove item (ii). Let $ \lambda^* $ satisfy the conclusion given by
Theorem~\ref{thm:hardcases} under conditions  $\lambda_{\min}(H) \le 0$, $g\in {\rm Range}(H - \lambda_{\min}(H)I)$ and
$\|(H - \lambda_{\min}(H) I)^\dagger g\| >  \sqrt{(\rho^+)_+'(-\lambda_{\min}(H))}$. It follows that $\lambda^* < \lambda_{\min}(H) \le 0$ and
\[
\|(H - \lambda^* I)^{-1} g\|^2 =  (\rho^+)'(-\lambda^*).
\]
Then following the same lines of arguments used in the proof of item (i) above, we get $ \lambda(t^*)=\lambda^*<\lambda_{\min}(H) $ and the differentiability of $ \widehat{k} $ follows from Remark~\ref{rem3.3}. \qed
\end{proof}

\subsection{RW algorithm for $\rho$RS}

The RW algorithm for solving $\rho$RS \eqref{Prob_general} with a $\rho$
satisfying Assumptions~\ref{assump:one}-\ref{assump:three} is presented in Algorithm~\ref{RW_rho}, which is
essentially a routine for solving the RW-dual \eqref{rw_dual0}; we refer the readers to \cite[Section~3.3]{PongWolk:12} for the detailed descriptions of vertical cut, triangle interpolation and inverse linear interpolation. Note
that we ruled out the case that $\widehat k$ is nondifferentiable at its
maximizer in Step 0, in view of Theorems~\ref{thm3} and \ref{thm4}. Thus, when a maximizer is found, one can recover a
solution of \eqref{Prob_general} according to Theorem~\ref{thm2}.\\

\begin{algorithm}[h]
\caption{RW$_\rho$: Rendl-Wolkowicz algorithm for $\rho$RS \eqref{Prob_general}}\label{RW_rho}
\begin{algorithmic}
\STATE
\begin{enumerate}
\item \underline{\bf Initialization}: If ``$H\succ 0$" or ``$g\in {\rm
Range}(H - \lambda_{\min}(H)I)$ and $\lambda_{\min}(H) \le 0$", perform
the case check described in Theorem~\ref{thm:hardcases} and obtain an explicit solution if possible.
\item \underline{\bf Main loop}: Now, according to Theorem~\ref{thm1} item \eqref{item:theoremiv} and Theorems~\ref{thm3} and \ref{thm4},
$\widehat k$ must be differentiable at any maximizer $t^*$ of \eqref{rw_dual0}. 
\begin{enumerate}[{\rm (i)}]
  \item Find an initial interval $I$ containing $t^*$ as in \eqref{t_bd0}.
  \item Update $t$.
  \begin{itemize}
    \item Compute the midpoint of $I$ as a candidate for updating $t$.
    \item If points at which $t\mapsto \widehat k(t)-t$ has positive and negative supergradients, respectively, are known, then:
    \begin{enumerate}
      \item Perform {\bf vertical cut} and reduce $I$ if possible. Update $t$ if vertical cut is performed successfully (i.e.,~yielding a point in $I$).
      \item Perform {\bf triangle interpolation} and update $t$ if triangle interpolation is successful.
    \end{enumerate}
    \item Perform {\bf inverse linear interpolation} on the equation
$(\widehat k^\prime(t))^{-\frac14} = 1$, and update $t$ if this step is successful.
  \end{itemize}
  \item Compute the function value and a supergradient of $\widehat k$
at $t$ as described in Theorem~\ref{thm1}. At points of differentiability, also
generate a primal iterate as described in Theorem~\ref{thm2}.
  \item Repeat Steps (i), (ii) and (iii) if a termination criterion is not met.
\end{enumerate}
\end{enumerate}
	\end{algorithmic}
\end{algorithm}

\section{Specific instances: $p$-regularization subproblem and beyond}\label{sec:pRS}

\subsection{$ p $-regularization subproblem}\label{subsec:prs}

We now specialize the results in Sections~\ref{sec:LD_opt} and \ref{sec:RW_dual} to derive the Lagrange dual and RW-dual for $p$RS in \eqref{Prob_prs}.
Note that this is a special case of
\eqref{Prob_general} with $\rho(t) = \frac{M}{p}t_+^{\frac{p}{2}}$: one can
check that both Assumptions~\ref{assump:one} and \ref{assump:two} are satisfied. Moreover, it is routine to compute that
\begin{equation}\label{rhoplus_prs}
\rho^+(u) = \sup_{t\ge 0}\left\{ut - \frac{M}{p}t_+^{\frac{p}2}\right\} = \frac{(p-2)M}{2p}\left(\frac{2u_+}{M}\right)^{\frac{p}{p-2}},
\end{equation}
showing that Assumption~\ref{assump:three} also holds. Now, the \textit{Lagrange dual} problem is given as in \eqref{L_dual_gen} with the $\rho^+$ in \eqref{rhoplus_prs}.
Furthermore, according to \eqref{rw_dual0} and \eqref{kt10}, the \textit{RW-dual} for $p$RS is the following maximization problem:
\begin{equation}\label{rw_dual}
\sup_t\ \widetilde k(t) - t,
\end{equation}
with
\begin{equation}\label{kt1}
\widetilde k(t) = \inf_{\gamma \ge 0}\left\{\frac{M}{p}(\gamma-1)_+^\frac{p}2 + \gamma \lambda_{\min}\left(\begin{bmatrix}
t& g^T\\g& H
\end{bmatrix}\right)\right\}.
\end{equation}
By Theorem~\ref{thm:rwdual}, the optimal value of the RW-dual defined in
\eqref{rw_dual} above is also $v_{_{p{\rm RS}}}$, and the supremum is attained.


In the {\bf RW}$_\rho$ algorithm, we use \eqref{t_bd0} to initialize the algorithm and compute the function value $\widehat k$ in each iteration. We now discuss how these can be done efficiently when specialized to $p$RS in \eqref{Prob_prs}. Indeed, for any $\alpha\ge 0$ satisfying $\rho(\alpha^2)\le 2\max\{2\|g\|\alpha,\|H\|\alpha^2\}$, we have
\[
\frac{M}{p}\alpha^p\le 2\max\{2\|g\|\alpha, \|H\|\alpha^2\}.
\]
This implies $\alpha\le \max\left\{\left(\frac{4p\|g\|}{M}\right)^\frac{1}{p-1},\left(\frac{2p\|H\|}{M}\right)^\frac1{p-2}\right\}$.
It hence follows that
\[
\tau\le \max\left\{\left(\frac{4p\|g\|}{M}\right)^\frac{1}{p-1},\left(\frac{2p\|H\|}{M}\right)^\frac1{p-2}\right\}.
\]
Similarly, for any $\beta > \max\{\lambda_{\min}(H),0\}$ satisfying $\beta\sqrt{(\rho^+)'(\beta - \lambda_{\min}(H))} \le \|g\|$, we have
\[
\|g\|\ge \beta \left(\frac{2(\beta - \lambda_{\min}(H))}{M}\right)^\frac1{p-2},
\]
which is the same as $2\beta^{p-1}-2\lambda_{\min}(H)\beta^{p-2} \le M\|g\|^{p-2}$. Let $\omega(\beta) := 2\beta^{p-1}-2\lambda_{\min}(H)\beta^{p-2} - M\|g\|^{p-2}$. Then one can check that $\omega(\beta) = 0$ has a unique root $\bar\beta$ satisfying $\bar \beta > \max\{\lambda_{\min}(H),0\}$. Moreover, for a nonnegative $\beta$, it holds that $\omega(\beta)\le 0$ if and only if $\beta\le \bar\beta$.
Thus, it follows that $\kappa\le \bar{\beta}$. Now we give an easily computable upper bound $\widehat\beta$ of $\bar\beta$ as follows:
\begin{itemize}
  \item When $ \lambda_{\min}(H)>0 $, we can deduce from $ 2\bar{\beta}^{p-2}(\bar{\beta}-\lambda_{\min}(H))=M\|g\|^{p-2} $ that $ 2[\lambda_{\min}(H)]^{p-2}(\bar{\beta}-\lambda_{\min}(H))<  M\|g\|^{p-2} $. It means that we can choose $ \widehat\beta = \frac{M\|g\|^{p-2}}{2[\lambda_{\min}(H)]^{p-2}} + \lambda_{\min}(H) $.
  \item When $ \lambda_{\min}(H)\leq 0 $, we have from $ \omega(\bar{\beta})=0 $ that $ 2\bar{\beta}^{p-1}\leq M\|g\|^{p-2} $. It means that we can choose $ \widehat\beta =  \left(\frac{M}{2\|g\|}\right)^{\frac1{p-1}}\|g\| $.
\end{itemize}
Hence, it holds that
\[
\kappa\le \widehat \beta := \begin{cases}
\frac{M}2 \left(\frac{\|g\|}{\lambda_{\min}(H)}\right)^{p-2} + \lambda_{\min}(H) & {\rm if }\ \lambda_{\min}(H) > 0,\\
\left(\frac{M}{2\|g\|}\right)^{\frac1{p-1}}\|g\| & {\rm otherwise}.
\end{cases}
\]

Consequently, we have from \eqref{t_bd0} and the above upper bounds for $\tau$ and $\kappa$ that if $t^*$ is a maximizer of \eqref{rw_dual} such that $\widetilde k$ is differentiable at $t^*$, then
\begin{equation}\label{t_bd}
\begin{aligned}
&\lambda_{\min}(H)-\widehat{\beta} \le t^*\\
&\le \min\left\{\bar t,\lambda_{\min}(H) + \|g\| \max\left\{\left(\frac{4p\|g\|}{M}\right)^\frac{1}{p-1},\left(\frac{2p\|H\|}{M}\right)^\frac1{p-2}\right\} \right\}.
\end{aligned}
\end{equation}
This bound can be used in place of \eqref{t_bd0} in {\bf RW}$_\rho$ when solving $p$RS.

Finally, as for evaluating $\widetilde k$ in \eqref{kt1} in each iteration,
one has to compute at where the infimum in \eqref{kt1} is attained. This
can be done by specializing Theorem~\ref{thm1} item \eqref{item:itheorem} to $p$RS: the set of minimizers for the infimum in \eqref{kt1} is given by
\[
\begin{cases}
\{0\} & {\rm if}\ \lambda(t) > 0,\\
[0,1] & {\rm if}\ \lambda(t) = 0,\\
\left\{\left(-\frac{2\lambda(t)}{M}\right)^\frac{2}{p-2} + 1\right\} & {\rm otherwise}.
\end{cases}
\]

\subsection{Combining trust region constraint and $p$-regularization}\label{subsec:stp}
We consider the following regularized problem, which combines \eqref{Prob_prs} and \eqref{Prob_trs}:
\begin{equation}\label{Prob_stp}
\begin{array}{lc}
(p{\rm TRS})&
\begin{array}{crl}
v_{_{p{\rm TRS}}} :=& \min\limits_x & 2g^Tx+x^THx+\frac{M}{p}\|x\|^p\\
&{\rm s.t.} & \|x\|^2\leq s,
\end{array}
\end{array}
\end{equation}
where $H\in {\cal S}^n$, $g\in \R^n\backslash\{0\}$, $ p>2 $,  $M > 0$ and $ s>0 $.
This is a special case of \eqref{Prob_general} with
\begin{equation}\label{phi}
\rho(t) =\frac{M}{p}t_+^{\frac{p}2}+\delta_{(-\infty,s]}(t).
\end{equation}
With the definition of monotone conjugate, one can compute that
\begin{equation}\label{phiplus}
\begin{aligned}
\rho^+(u) &:= \sup\limits_{t\geq 0}\left\{ut-\frac{M}{p}t_+^{\frac{p}2}-\delta_{(-\infty,s]}(t)\right\}\\
& =
\begin{cases}
 \frac{(p-2)M}{2p}\left(\frac{2u_+}{M}\right)^{\frac{p}{p-2}} &{\rm if}\ u\leq \frac{M}2s^{\frac{p-2}2},\\
us-\frac{M}{p}s^{\frac p 2}& {\rm otherwise}.
\end{cases}
\end{aligned}
\end{equation}
It is now routine to check that the $\rho$ in \eqref{phi} satisfies Assumptions~\ref{assump:one}-\ref{assump:three}. The \textit{Lagrange dual} problem for \eqref{Prob_stp} is given as in \eqref{L_dual_gen} with the $\rho^+$ in \eqref{phiplus}. Furthermore, the \textit{RW-dual} for \eqref{Prob_stp} in reference to \eqref{rw_dual0} and \eqref{kt10} is
\begin{equation}\label{rw_dual_stp}
\sup\limits_{t} \  \breve k(t)-t
\end{equation}
where
\begin{equation}\label{kt_stp}
\breve k(t) = \inf\limits_{0\leq\gamma\leq s+1} \left\{\frac{M}{p}(\gamma-1)_+^{\frac p 2}+ \gamma \lambda_{\min}\left(\begin{bmatrix}
t& g^T\\g& H
\end{bmatrix}\right)\right\}
\end{equation}
By Theorem~\ref{thm:rwdual}, the optimal value in \eqref{rw_dual_stp} is also $v_{_{p{\rm TRS}}}$ and the supremum is attained.

Now, we estimate upper bounds for the quantities $ \tau $ and $ \kappa $ in \eqref{tau} and \eqref{kappa} respectively, which are used in $ {\rm RW}_\rho $ algorithm for solving \eqref{Prob_stp}. According to the definition of $ \tau $, we consider those nonnegative $ \alpha $'s satisfying $ \rho(\alpha^2)\leq2\max\{2\|g\|\alpha,\|H\|\alpha\} $. Using the definition of $\rho$ in \eqref{phi}, we obtain
\[
0\leq \alpha \leq \sqrt{s} \ {\rm and}\ \frac{M}{p}\alpha^p\le 2\max\{2\|g\|\alpha, \|H\|\alpha^2\}.
\]
This implies $ \alpha \leq \min\left\{\sqrt{s}, \ \max\left\{\left(\frac{4p\|g\|}{M}\right)^\frac{1}{p-1},\left(\frac{2p\|H\|}{M}\right)^\frac1{p-2}\right\}\right\} $.
Thus, we have
\[
\tau \leq  \widetilde\alpha:=\min\left\{\sqrt{s}, \ \max\left\{\left(\frac{4p\|g\|}{M}\right)^\frac{1}{p-1},\left(\frac{2p\|H\|}{M}\right)^\frac1{p-2}\right\}\right\}.
\]

Next, to upper estimate $ \kappa $, it suffices to find a $\widetilde\beta\ge 0$ so that
\[
\widetilde{\beta}\ge \sup\left\{\beta > 0:\;\beta>\lambda_{\min}(H),\  \beta\sqrt{(\rho^+)^\prime(\beta-\lambda_{\min}(H))}\leq \|g\|\right\},
\]
where
\begin{equation}\label{drhoplus_stp}
 (\rho^+)^\prime(u) = \begin{cases}
 \left(\frac{2u_+}{M}\right)^\frac2{p-2} & {\rm if}\  u\leq \frac M 2 s^{\frac{p-2}2},\\
 s & {\rm otherwise}
 \end{cases}
\end{equation}
To this end, we start by noting that if $\beta-\lambda_{\min}(H)>\frac{M}2s^{\frac{p-2}2}$ and $\beta > 0$, then we have from $\beta\sqrt{(\rho^+)^\prime(\beta-\lambda_{\min}(H))}\leq \|g\|$ that $\beta\le \frac{\|g\|}{\sqrt{s}}$. Otherwise, we have $\beta-\lambda_{\min}(H)\le\frac{M}2s^{\frac{p-2}2}$. Thus, one can choose $\widetilde\beta$ as
\[
\widetilde \beta := \max\left\{\frac{\|g\|}{\sqrt{s}}, \lambda_{\min}(H) + \frac{M}2s^{\frac{p-2}2}\right\}.
\]
Then it holds that $\kappa \le \widetilde \beta$.

Now, we have from Theorem~\ref{thm2} and the upper bounds for $ \tau $ and $ \kappa $ that, if $ t^* $ is a maximizer of \eqref{rw_dual_stp} such that $\breve k$  is differentiable, then
\begin{equation}\label{t_bd_stp}
\lambda_{\min}(H)-\widetilde{\beta} \leq t^* \leq \min\left\{\bar t, \lambda_{\min}(H)+\widetilde{\alpha}\|g\| \right\}.
\end{equation}
Finally, we turn to the infimum in \eqref{kt_stp} that will be used for evaluating $ \breve k $ in each iteration. Thanks to Theorem~\ref{thm1} item (i) and \eqref{drhoplus_stp}, we obtain the set of $\gamma$ that attains the infimum in \eqref{kt_stp} as follows,
 \[
 \begin{cases}
 \{0\} & {\rm if }\ \lambda(t)>0,\\
 [0, 1] & {\rm if }\ \lambda(t) = 0,\\
 \left\{\left(-\frac{2\lambda(t)}{M}\right)^\frac{2}{p-2} + 1\right\} & {\rm if }\ -\frac M 2 s^{\frac{p-2}2} \leq \lambda(t) <0,\\
 \{s+1\} & {\rm otherwise}.
 \end{cases}
 \]

\subsection{Concrete examples of $\rho$ verifying Assumption~\ref{assump:three}}\label{subsec:rho}

We present in the next proposition a sufficient condition for a
proper closed convex function $\rho$ satisfying
Assumptions~\ref{assump:one} and \ref{assump:two} to also satisfy Assumption~\ref{assump:three}. For convenience, we denote $ {\frak D} :=\{t\in{\rm dom}\,\rho :\; \rho(t)>0 \} $.

\begin{prop}[Sufficient condition for Assumption~\ref{assump:three}]
  Let $\rho$ be a proper closed convex function satisfying
Assumptions~\ref{assump:one} and \ref{assump:two}. If $ \rho $ is strictly convex on $ \frak D $, then $\rho^+$ is differentiable on $(0,\infty)$.
\end{prop}
\begin{proof}
Since $ \rho $ is proper closed convex and satisfies Assumption~\ref{assump:one}, the set $ \frak D $ is convex. We consider two cases: (i) $ \frak D = \emptyset $;  (ii) $ \frak D  $ is a nonempty convex set.

In case (i), the value of $ \rho $ is either 0 or $ \infty $. According to Assumptions~\ref{assump:one} and \ref{assump:two}, we have $ \rho(0) = 0 $, $\rho(t_0) <\infty$ for some $t_0 > 0$, and $ \rho(t)\rightarrow\infty $ as $ t\rightarrow \infty $, it means that there exists $ \eta>0 $ such that $ \rho(t) = \delta_{(-\infty,\eta]} $. One can then compute
\[
\rho^+(u) = \sup\limits_{0\leq t\leq \eta} tu =  \begin{cases}
\eta u & {\rm if }\ u\geq 0, \\
0 & {\rm otherwise}.
\end{cases}
\]
So the claim holds trivially for this case.

We now consider case (ii). Let $ a := \inf\{ t: t\in\frak D \} $. Then we have $ a\in[0, \infty) $, $\rho(t) \in (0,\infty]$ whenever $t > a$, and $\rho(t) = 0$ when $t < a$, thanks to Assumption~\ref{assump:one}. In addition, since ${\frak D}\neq \emptyset$, we conclude further that $a\in {\rm dom}\,\rho$. Using these and the fact that $\rho$ is closed and convex, we see from \cite[Proposition~2.1.6]{Za02} that $\rho(a) = 0$. This together with ${\frak D}\neq \emptyset$ implies that $a\in {\rm int}({\rm dom}\,\rho)$.

Since $\rho(t)=0$ whenever $t \in [0,a]$, one can see that $ a \in \mathop{\rm argmax}\limits_{0\leq t\leq a}\{ut-\rho(t)\} $ for $u\in \R_+$. Therefore, for any $ u\in\R_+ $, we have
\begin{equation}\label{gjr0}
\begin{aligned}
\rho^+(u) &= \sup\limits_{t\geq 0} \{ut-\rho(t)\} = \sup\limits_{t\geq a} \{ut-\rho(t)\} \\
&= \sup\limits_{t}\{ut-(\rho+\delta_{[a,\infty)})(t)\} = (\rho+\delta_{[a,\infty)})^*(u),
\end{aligned}
\end{equation}
which together with the fact $ {\rm dom}\, \rho^+ = \R $ (consequence of Assumption~\ref{assump:two}) implies $\R_+ \subseteq {\rm dom}\,(\rho+\delta_{[a,\infty)})^*  $. So we have
\begin{equation}\label{gjr1}
	(0,\infty) \subseteq {\rm int}\, {\rm dom}(\rho+\delta_{[a,\infty)})^*.
\end{equation}
Since $ \rho $ is strictly convex on $ \frak D $ and $ {\rm int}\, {\rm dom}\, (\rho+\delta_{[a,\infty)}) \subseteq \frak D $, the function $ \rho+\delta_{[a,\infty)} $ is strictly convex. In view of \cite[Theorem 4.2.5]{BoLe:00}, one can then deduce that $ (\rho+\delta_{[a,\infty)})^* $ is essentially smooth. Thus, $ (\rho+\delta_{[a,\infty)})^* $ is differentiable throughout $ {\rm int}\, {\rm dom}\,(\rho+\delta_{[a,\infty)})^* $. Then we know from \eqref{gjr0} and \eqref{gjr1} that $ \rho^+ $ is differentiable on $ (0,\infty) $.
%
%
This completes the proof. \qed
\end{proof}

\section{Numerical experiments on generic instances for $p$RS and $p$TRS}\label{sec4}
In this section, we present numerical experiments to compare our algorithm against the standard approach of using Newton's method (with backtracking) to solve the Lagrange dual problems of $p$RS \eqref{Prob_prs} and $p$TRS \eqref{Prob_stp}. In the case when $p=3$ in $p$RS, i.e., the cubic-regularization subproblem, we also consider the recently proposed generalized eigenvalue based approach (GEP) in \cite[Algorithm~A3]{Lieder19}. All numerical experiments are performed in MATLAB 2019b on a 64-bit PC with an Intel Core i7-6700 CPU (3.40GHz) and 32GB of RAM.
Specifically, we compare the following algorithms:

\begin{enumerate}[{\rm (i)}]
  \item {\bf GEP}: This solves $p$RS \eqref{Prob_prs} for $p = 3$ (i.e., the cubic-regularization subproblem) by computing the largest real generalized eigenvalue $ \lambda^*\geq 0 $ and the corresponding eigenvector $ v $ of a certain generalized eigenvalue problem (see \cite[Algorithm~A2, Step~2]{Lieder19}), which can be transformed to a standard eigenvalue problem (\cite[Equation~(18)]{Lieder19} with $ \sigma = M/2 $).
  In our implementation of \cite[Algorithm~A3]{Lieder19}, we solve the eigenvalue problem involved by the MATLAB command {\sf eigs} with default tolerance. We also set the tolerance $ \delta = 10^{-5} $ in \cite[Algorithm~A3]{Lieder19}. Furthermore, we skip the optional steps in Step 4 of \cite[Algorithm~A3]{Lieder19}, and for the {\sf minresQLP} algorithm\footnote{https://web.stanford.edu/group/SOL/software/minresqlp/} used in Step~4a) there, we use the default settings.

  \item {\bf Newton$ _\rho $}: This uses the Newton's method\footnote{Note that the function $ \rho^+ $ with respect to $p$TRS in \eqref{phiplus} is not twice differentiable at $ \widehat u = \frac{M}{2}s^{\frac{p-2}2} $. In the implementation of Newton's method for solving \eqref{Prob_stp}, we use $\frac{4}{M(p-2)}s^{\frac{4-p}{2}} $ as ``Hessian" at $\widehat u$.} with (Armijo) line search for solving the Lagrange dual problems of \eqref{Prob_prs} or \eqref{Prob_stp}. The specific form of these two Lagrange dual problems are described in Sections~\ref{subsec:prs} and \ref{subsec:stp} respectively. As a heuristic, we initialize the algorithm at $ \lambda = \min\{0, \lambda_{\min}(H)\}-1 $. The $\lambda_{\min}(H)$ is computed using the MATLAB command {\sf eigs} with {\sf opts.issym = 1}, {\sf opts.maxit = 5000}, {\sf opts.v0 = sum(H)'}, {\sf opts.fail = 'keep'} and {\sf opts.tol = 1e-8}. The quantity $ (H-\lambda I)^{-1}v $ that appears in the gradient and the Hessian computation is obtained via the MATLAB command {\sf pcg}, using default tolerance and a maximum number of iterations of {\sf 5000}; we also warm-start using approximate solution from the previous iteration. A simple backtracking strategy is then applied to ensure sufficient ascent of the objectives of the dual problems and the positive definiteness of $ H-\lambda I $. We terminate the algorithm when
  \[
  \frac{|f(x) - \widetilde d(\lambda)|}{|f(x)| + 1} < 10^{-9} \ {\rm with}\ x = {\rm Proj}_{\left\{w:\; \|w\|^2\in {\rm dom}\,\rho\right\}}\left(-(H-\lambda I)^{-1}g\right),\footnote{Note that for $p$RS, ${\rm dom}\,\rho = \R$, while for $p$TRS, ${\rm dom}\,\rho = (-\infty,s]$, we see that in both cases the set $\left\{x:\; \|x\|^2\in {\rm dom}\,\rho\right\}$ is closed convex and nonempty. Thus, the projection onto this set exists and is unique.}
  \]
  or the stepsize falls below $ 10^{-10} $ or the number of iterations reaches 10; here $\widetilde d$ is the objective function of \eqref{L_dual_gen}.; in addition, we also terminate if $ |\lambda -\lambda_{\min}(H)|<10^{-10} $ and $ \frac{\|x\|-\sqrt{(\rho^+)^\prime_+(-\lambda_{\min}(H))}}{\sqrt{(\rho^+)^\prime_+(-\lambda_{\min}(H))} + 1} \le 10^{-10}$ --- This latter case suggests the hard case 2 likely occurs. In this latter case, we output $ x + \alpha v $ as an approximate solution, where $ \alpha\in\R $ and $ v\in {\rm ker}(H-\lambda_{\min}(H) I)\setminus\{0\} $ satisfy $ \|x+\alpha v\|^2 = (\rho^+)^\prime_+(-\lambda_{\min}(H)) $.

  \item {\bf RW$ _\rho $}: This is an implementation of {\bf RW}$_\rho$ for $p$RS and $p$TRS. To perform the case check, we first compute $(\lambda_{\min}(H), v^*)$ using {\sf eigs} with the same settings as in {\bf Newton$_\rho $}. If $\lambda_{\min}(H) > 0$, we return zero as the solution. On the other hand, if $\lambda_{\min}(H)\le 0$ and $|{v^*}^Tg| < 10^{-8}\|g\|$, we apply {\sf minresQLP} algorithm with {\sf rtol = 1e-10} to obtain approximately $-(H - \lambda_{\min}(H)I)^\dagger g$; if the residual\footnote{We use the output {\sf relres} from {\sf minresQLP}.} is small ($< 10^{-8}$), we further check whether we are in {\em hard case 2}; we compute an explicit solution as described in \eqref{xhardcase} when {\em hard case 2} happens.

      Otherwise, we initialize the algorithm using the intervals given in \eqref{t_bd} and \eqref{t_bd_stp} for $p$RS and $p$TRS respectively, and terminate when
   \[
   \frac{|f(x) - (\widehat k(t) - t)|}{|f(x)| + 1} < 10^{-12} \ {\rm with}\ x = \frac{\widetilde x\sqrt{(\rho^+)^\prime(-\lambda(t))}}{\|\widetilde x\|}
   \]
  or when the width of the interval $I:= [a,b]$ is too small, i.e., $\frac{|b-a|}{|a| + |b|} < 10^{-12}$; here, $\widehat k$ is given in \eqref{kt10}, and $ \rho^+ $ is given in \eqref{rhoplus_prs} and \eqref{phiplus} respectively for $p$RS and $p$TRS, and $\widetilde x = \frac{1}{v_0}\bar v$ with $\begin{bmatrix}
     v_0& \bar v^T
   \end{bmatrix}^T$ being a unit eigenvector corresponding to the
$\lambda(t)$ (see \eqref{define}) at the current iteration (this is motivated by the formula
for $x^*$ in Theorem~\ref{thm2}). The $\|H\|$ in \eqref{t_bd} and \eqref{t_bd_stp} is computed by using the MATLAB command {\sf normest} with a tolerance $ 10^{-2} $. Moreover, for computational efficiency, the eigenvalue $\lambda(t)$ is computed using {\sf eigs} with a tolerance $10^{-8} $,
and we also warm-start from an eigenvector obtained from a previous iteration at which the slope of $w\mapsto \widehat k(w)-w$ has the same sign as that at $t$.\footnote{For the first iteration, we initialize {\sf eigs} at $[1\ v^*]^T$, where $v^*$ is an eigenvector of $H$ corresponding to $\lambda_{\min}(H)$ obtained during the case check in Step 1.} 

\end{enumerate}

Below, we present tabulated numerical results on randomly generated instances of \eqref{Prob_prs} with $ p=3$, $p=3.5 $, and instances of \eqref{Prob_stp} with  $ (p,s) = (3, 10) $ respectively. We consider three cases:

\paragraph{Easy case.} These are generic instances. To generate such instances, we randomly generate $g\in \R^n$ in \eqref{Prob_prs} or \eqref{Prob_stp} with i.i.d. Gaussian entries, and generate $H$ using the MATLAB command
{\sf H = sprandsym(n,0.005)}. The $M$'s in \eqref{Prob_prs} and \eqref{Prob_stp} are set to be $1.2\|H\|$.

\paragraph{Hard case 1.} This means that the case check in Step 1 is active, but we are not necessarily in hard case 2.
To generate such instances, we randomly generate a sparse $H$ and set $ M $ as in \emph{Easy case}, except that we discard the $H$ if $\lambda_{\min}(H)\ge 0$ and regenerate $H$. We then set the $ g $ in \eqref{Prob_prs} as $g = (H - \lambda_{\min}(H)I)v$ for some $v\in \R^n$ generated according to the MATLAB codes below:
\begin{verbatim}
u = randn(n,1);
v = H*u - lambdamin*u;
v = 1.1*v/norm(v)*(-2/M*lambdamin)^(1/(p-2));
\end{verbatim}
where {\sf lambdamin} is $ \lambda_{\min}(H) $. Similarly, the $ g $ in \eqref{Prob_stp} is set as $g = (H - \lambda_{\min}(H)I)v$  using a $ v $ generated according to the MATLAB codes below:
\begin{verbatim}
u = randn(n,1);
v = H*u - lambdamin*u;
v = 1.1*v/norm(v)*sqrt(dphiplus_stp(-lambdamin,p,s,M));
\end{verbatim}
where {\sf dphiplus\_stp} computes the $ (\rho^+)^\prime $ in \eqref{drhoplus_stp}. Then we are likely in {\em Hard case 1}.

\paragraph{Hard case 2.} This means that the case check in Step 1 is active and the optimal solution $ x^* $ can be obtained explicitly. To generate such instances of \eqref{Prob_prs} and \eqref{Prob_stp}, we randomly generate a sparse $ H $ and set $ M $ as in {\em Hard case 1}. Then we generate $ g $ in a way similar to that in {\em Hard case 1} except that we use a coefficient of 0.9 instead of 1.1 in the MATLAB codes for $ v $.

In our experiments below, we set $n = 25000i$ for $ i= 1,2,3, 4 $. The computational results of $p$RS are shown in Tables~\ref{Table1} and \ref{Table2}, which correspond to $ p=3$, $p=3.5$, respectively. The computational results of $p$TRS are presented in Tables~\ref{Table3}, where  $ (p,s) = (3, 10) $. We report, averaged over $20$ random instances, the CPU time, the number of iterations, and the quantity
\[
{\rm ratio}_i := \frac{f_i - f_{\min}}{|f_{\min}|},
\]
where $f_i$, $i=$ {\bf GEP}, {\bf Newton}$_\rho$ or {\bf RW}$_\rho$, stands for the terminating objective value obtained by each algorithm on a problem instance, and $f_{\min}$ is the smallest value obtained among the competing algorithms on that problem instance. 
Comparing {\bf RW}$_\rho$ and {\bf Newton}$_\rho$, {\bf RW}$_\rho$ always gives a better ratio and is always faster than {\bf Newton}$_\rho$ on \textit{Easy case} and \textit{Hard Case 2} instances, but is slower otherwise. Furthermore, for $p$RS with $p = 3$ where {\bf GEP} is applicable, {\bf GEP} outperforms other approaches on \textit{Easy case} and \textit{Hard case 1} instances, but is slower than {\bf RW}$_\rho$ on \textit{Hard Case 2} instances. However, it is not known whether the $\rho$-regularization subproblem \eqref{Prob_general} with general $\rho$ can be solved via a similar approach.

\begin{table}[h]
	\caption{Random tests for $p$RS with $ p=3 $}\label{Table1}
{\scriptsize
			\begin{tabular}{lc rc rc rc}\hline
	 & & \multicolumn{2}{c}{GEP} &  \multicolumn{2}{c}{Newton$_\rho$} & \multicolumn{2}{c}{RW$_\rho$} \\
\cline{3-8}
	~&	$n$ &  CPU  & ${\rm ratio}_{GEP}$  &  CPU(iter)  & ${\rm ratio}_{{\rm Newton}_\rho}$ & CPU(iter)  & ${\rm ratio}_{{\rm RW}_\rho}$  \\
   Easy case  &  25000 &   0.4 & 1.4e-16 &    3.4( 7) & 2.5e-13 &   2.4( 6) & 6.0e-17 \\
              & 50000  &   2.1 & 4.3e-17 &   20.3( 7) & 4.9e-11 &  14.9( 6) & 1.4e-16 \\
              & 75000  &   5.1 & 6.1e-17 &   61.3( 8) & 3.9e-11 &  45.1( 6) & 8.7e-17 \\
              & 100000 &   9.4 & 1.1e-16 &  120.8( 8) & 2.7e-13 &  88.1( 6) & 1.3e-16 \\
Hard case 1   & 25000  &   1.9 & 1.4e-16 &    3.6(10) & 1.1e-12 &   3.8( 8) & 4.5e-17 \\
              & 50000  &   8.7 & 1.7e-16 &   15.1( 2) & 2.0e-12 &  22.4( 8) & 8.0e-17 \\
              & 75000  &  21.4 & 9.1e-17 &   42.4( 2) & 2.0e-12 &  60.2( 8) & 1.7e-16 \\
              & 100000 &  39.2 & 9.0e-17 &   88.3( 2) & 1.2e-12 & 120.5( 8) & 1.5e-16 \\
Hard case 2   & 25000  &  13.7 & 2.6e-10 &    3.8( 8) & 8.0e-12 &   2.7( 0) & 0.0e+00 \\
              & 50000  &  75.3 & 2.6e-10 &   20.7( 8) & 9.2e-12 &  15.6( 0) & 0.0e+00 \\
              & 75000  & 217.6 & 2.6e-10 &   58.0( 8) & 1.0e-11 &  45.8( 0) & 0.0e+00 \\
              & 100000 & 464.1 & 2.7e-10 &  115.9( 8) & 9.3e-12 &  93.8( 0) & 0.0e+00 \\
				\hline
			\end{tabular}
}
\end{table}
\vspace{-2mm}
\begin{table}[h]
	\caption{Random tests for $p$RS with $ p=3.5 $}\label{Table2}
		{\scriptsize
			\begin{tabular}{lc rc rc}\hline
 &  &\multicolumn{2}{c}{Newton$_\rho$} & \multicolumn{2}{c}{RW$_\rho$}   \\
\cline{3-6}
~&	$n$  &  CPU(iter)  & ${\rm ratio}_{{\rm Newton}_\rho}$ & CPU(iter)  & ${\rm ratio}_{{\rm RW}_\rho}$  \\
    Easy case  &  25000  &    3.5( 8) & 1.5e-13 &   2.5( 6) & 0.0e+00 \\
               & 50000   &   21.3( 8) & 8.8e-12 &  15.4( 6) & 0.0e+00 \\
               & 75000   &   57.6( 8) & 1.7e-10 &  41.1( 6) & 0.0e+00 \\
               & 100000  &  118.7( 9) & 3.1e-11 &  84.3( 6) & 0.0e+00 \\
 Hard case 1   & 25000   &    2.3( 2) & 1.7e-12 &   3.9( 8) & 0.0e+00 \\
               & 50000   &   15.0( 2) & 1.5e-12 &  22.3( 8) & 0.0e+00 \\
               & 75000   &   40.3( 2) & 3.1e-12 &  58.2( 8) & 0.0e+00 \\
               & 100000  &   88.0( 2) & 2.6e-12 & 120.5( 8) & 0.0e+00 \\
 Hard case 2   & 25000   &    3.8( 8) & 7.5e-12 &   2.6( 0) & 0.0e+00 \\
               & 50000   &   21.3( 8) & 8.9e-12 &  16.4( 0) & 0.0e+00 \\
               & 75000   &   58.5( 8) & 9.4e-12 &  46.3( 0) & 0.0e+00 \\
               & 100000  &  115.8( 8) & 8.6e-12 &  93.5( 0) & 0.0e+00 \\			
 \hline
			\end{tabular}
		}
\end{table}

\begin{table}[h]
	\caption{Random tests for $p$TRS with $ p=3 $ and $ s=10 $}\label{Table3}
		{\scriptsize
			\begin{tabular}{lc rc rc}\hline
 & & \multicolumn{2}{c}{Newton$_\rho$} & \multicolumn{2}{c}{RW$_\rho$}   \\
\cline{3-6}
~&	$n$  &  CPU(iter)  & ${\rm ratio}_{{\rm Newton}_\rho}$ & CPU(iter)  & ${\rm ratio}_{{\rm RW}_\rho}$  \\
   Easy case  &  25000  &    3.7( 8) & 3.9e-13  &   2.6( 6) & 0.0e+00 \\
              & 50000   &   21.4( 8) & 4.3e-12  &  15.3( 6) & 0.0e+00 \\
              & 75000   &   62.3( 8) & 1.6e-11  &  45.1( 6) & 0.0e+00 \\
              & 100000  &  129.2( 9) & 3.0e-13  &  94.2( 6) & 0.0e+00 \\
 Hard case 1  & 25000   &    3.7(10) & 1.1e-12  &   3.9( 8) & 0.0e+00 \\
              & 50000   &   14.8( 2) & 2.0e-12  &  22.3( 8) & 0.0e+00 \\
              & 75000   &   44.0( 2) & 2.0e-12  &  62.0( 8) & 0.0e+00 \\
              & 100000  &   89.7( 2) & 1.2e-12  & 122.8( 8) & 0.0e+00 \\
 Hard case 2  & 25000   &    3.7( 8) & 7.7e-12  &   2.6( 0) & 0.0e+00 \\
              & 50000   &   21.7( 8) & 9.2e-12  &  16.7( 0) & 0.0e+00 \\
              & 75000   &   56.9( 8) & 1.0e-11  &  44.7( 0) & 0.0e+00 \\
              & 100000  &  118.1( 8) & 9.5e-12  &  95.6( 0) & 0.0e+00 \\
				\hline
			\end{tabular}
		}
\end{table}

\begin{acknowledgements}
This work was initiated from a discussion between the second author and Henry Wolkowicz on algorithms for the cubic-regularization subproblem during the second author's visit at the University of Waterloo in 2018. The authors gratefully acknowledge the many stimulating suggestions and comments from Henry Wolkowicz during the preparation of this paper.
\end{acknowledgements}


\begin{thebibliography}{99}
\bibitem{Adachi}
S. Adachi, S. Iwata, Y. Nakatsukasa, and A. Takeda.
\newblock{Solving the trust-region subproblem by a generalized eigenvalue problem.}
\newblock{\em SIAM J. Optim.}, 27:269--291, 2017.

\bibitem{BoLe:00}
J.M. Borwein and A.S. Lewis.
\newblock{\em Convex Analysis and Nonlinear Optimization: Theory and Examples.}
\newblock{Springer-Verlag, New York, 2000.}

\bibitem{CarmonDuchi18}
Y. Carmon and J.C. Duchi.
\newblock{Analysis of Krylov subspace solutions of regularized nonconvex quadratic problems.}
\newblock{\em Neural Inf Process Syst.}, 10728--10738, 2018.

\bibitem{MR2776701}
C. Cartis, N.I.M. Gould, and P.L. Toint.
\newblock Adaptive cubic regularisation methods for unconstrained optimization. Part I: motivation, convergence and numerical results.
\newblock {\em Math. Program.}, 127:245--295, 2011.

\bibitem{ConGouToi:00}
A.R. Conn, N.I.M. Gould, and P.L. Toint.
\newblock {\em Trust-Region Methods.}
\newblock {Society for Industrial and Applied Mathematics}, Philadelphia, 2000.

\bibitem{FortinWolk:03}
C. Fortin and H. Wolkowicz.
\newblock{The trust region subproblem and semidefinite programming.}
\newblock{\em Optim. Methods Softw.}, 19:41--67, 2004.

\bibitem{GoLuRoTo}
N.I.M. Gould, S. Lucidi, M. Roma, and P.L. Toint.
\newblock{Solving the trust-region subproblem using the Lanczos method.}
\newblock{\em SIAM J. Optim.}, 9:504--525,1999.

\bibitem{GoRoTho10}
N.I.M. Gould, D.P. Robinson, and H.S. Thorne.
\newblock {On solving trust-region and other regularised subproblems in optimization.}
\newblock {\em Math. Program. Comput.}, 2:21--57, 2010.

\bibitem{GoSi20}
N.I.M. Gould and V. Simoncini.
\newblock {Error estimates for iterative algorithms for minimizing regularized quadratic subproblems.}
\newblock {\em Optim. Methods Softw.}, 35:304--328, 2020.

\bibitem{Griewank81}
A. Griewank.
\newblock {The modification of Newton's method for unconstrained optimization by bounding cubic terms.}
\newblock {Technical report}, NA/12, 1981.

\bibitem{Hager:00}
W.W. Hager.
\newblock {Minimizing a quadratic over a sphere.}
\newblock {\em SIAM J. Optim.}, 12:188--208, 2001.

\bibitem{HL}
J.B. Hiriart-Urruty and C. Lemarechal.
\newblock {\em Fundamentals of Convex Analysis.}
\newblock {Springer-Verlag}, New York, 1993.

\bibitem{HsiaSheu17}
Y. Hsia, R.-L. Sheu, and Y.-X. Yuan.
\newblock{Theory and application of $p$-regularized subproblems for $ p>2 $.}
\newblock{\em Optim. Methods Softw.}, 32:1059--1077, 2017.

\bibitem{Lieder19}
F. Lieder.
\newblock {Solving large scale cubic regularization by a generalized eigenvalue problem.}
\newblock{\em SIAM J. Optim.}, 30:3345--3358, 2020.

\bibitem{MoSo:83}
J.J. Mor\'{e} and D.C. Sorensen.
\newblock{Computing a trust region step.}
\newblock {\em SIAM J. Sci. Statist. Comput.}, 4:553--572, 1983.

\bibitem{MR2229459}
Y. Nesterov and B.T. Polyak.
\newblock Cubic regularization of Newton method and its global performance.
\newblock {\em Math. Program.}, 108: 177--205, 2006.

\bibitem{PongWolk:12}
T.K. Pong and H. Wolkowicz.
\newblock The generalized trust region subproblem.
\newblock {\em Comput. Optim. Appl.}, 58:273--322, 2014.

\bibitem{ReWo:94}
F. Rendl and H. Wolkowicz.
\newblock A semidefinite framework for trust region subproblems with applications to large scale minimization.
\newblock {\em Math. Program.}, 77:273--299, 1997.

\bibitem{roc70}
R.T. Rockafellar.
\newblock {\em Convex Analysis.}
\newblock {Princeton University Press}, 1970.

\bibitem{Sor:99}
M. Rojas, S.A. Santos, and D.C. Sorensen.
\newblock{A new matrix-free algorithm for the large-scale trust-region subproblem.}
\newblock {\em SIAM J. Optim.}, 11:611--646, 2000.

\bibitem{MR2401375}
M. Rojas, S.A. Santos, and D.C. Sorensen.
\newblock{Algorithm 873: LSTRS: MATLAB software for large-scale trust-region subproblems and regularization.}
\newblock {\em ACM Trans. Math. Software}, 34:1--28, 2008.

\bibitem{StWo:93}
R. Stern and H. Wolkowicz.
\newblock {Indefinite trust region subproblems and nonsymmetric eigenvalue perturbations.}
\newblock {\em SIAM J. Optim.}, 5:286--313, 1995.

\bibitem{Za02}
C. Zalines\c{c}u.
\newblock {\em Convex Analysis in General Vector Spaces}.
\newblock World Scientific, 2002.


\end{thebibliography}
\end{document}